\documentclass[12pt]{article}
\usepackage{color}
\usepackage{amsmath}
\usepackage[colorlinks=true]{hyperref}
\usepackage{microtype}
\usepackage{amscd,amsthm,amsfonts,amssymb,esint}
\usepackage{fullpage}
\usepackage[all]{xy}
\usepackage{mathtools}
\usepackage{xcolor}
\usepackage{mathrsfs}
\usepackage{amsmath}
\usepackage[all]{xy}
\usepackage{geometry}
\usepackage{array}
\usepackage{colonequals}
\usepackage{tikz-cd}
\usepackage{comment}
\usepackage{setspace}
\usepackage{enumerate}

\begin{document}
	\title{A geometric approach to the uniform boundedness of $\ell$-primary torsion points}

	\author{Zhuchao Ji, Jiarui Song, Junyi Xie}
	
	\date{}
	
	\maketitle
	
	\begin{abstract}
		We prove that for a non-isotrivial abelian scheme over a smooth curve, the genus of a generic sequence of multi-sections with small heights tends to infinity.
		
		As an application, we give a new proof of the uniform boundedness of $\ell$-primary torsion points on fibers of an abelian scheme over a smooth curve, a result originally proved by Cadoret and Tamagawa. Furthermore, our approach allows us to reprove a conjecture of Cadoret and Tamagawa without additional assumptions.
		
		Our approach is based on the theory of Betti foliations and the arithmetic equidistribution theorem. 
	\end{abstract}
	
	\theoremstyle{plain}
	\newtheorem{thm}{Theorem}[section]
	\newtheorem{theorem}[thm]{Theorem}
	\newtheorem{cor}[thm]{Corollary}
	\newtheorem{corollary}[thm]{Corollary}
	\newtheorem{lem}[thm]{Lemma}
	\newtheorem{lemma}[thm]{Lemma}
	\newtheorem{pro}[thm]{Proposition}
	\newtheorem{proposition}[thm]{Proposition}
	\newtheorem{prop}[thm]{Proposition}
	\newtheorem{assumption}[thm]{Assumption}
	\newtheorem{conjecture}[thm]{Conjecture}
	\newtheorem{sublemma}[thm]{Sub-lemma}
	
	\theoremstyle{definition}
	\newtheorem{definition}[thm]{Definition}
	
	\theoremstyle{remark}
	\newtheorem{remark}[thm]{Remark}
	\newtheorem{example}[thm]{Example}
	\newtheorem{remarks}[thm]{Remarks}
	\newtheorem{problem}[thm]{Problem}
	\newtheorem{exercise}[thm]{Exercise}
	\newtheorem{situation}[thm]{Situation}
	\newtheorem{Question}[thm]{Question}

	\newcommand\mO{\mathcal{O}}
	\newcommand{\mf}[1]{\mathfrak{#1}}
	\newcommand{\ms}[1]{\mathscr{#1}}
	\newcommand{\mb}[1]{\mathbb{#1}}
	\newcommand{\mr}[1]{\mathrm{#1}}
	\newcommand{\mc}[1]{\mathcal{#1}}
	\newcommand{\ov}{\overline}
	\newcommand{\Zar}[1]{\overline{#1}^{\mathrm{Zar}}}
	\newcommand\ra{\rightarrow}
	\newcommand{\ora}[1]{\stackrel{#1}\longrightarrow}
	\newcommand{\cra}{\stackrel{\sim}\rightarrow}
	\newcommand\mheight{\mathrm{height\,}}
	\newcommand\wt{\widetilde}
	\newcommand\im{\mathrm{im\,}}
	\newcommand\Hom{\mathrm{Hom}}
	\newcommand\Spe{\mathrm{Spec\,}}
	\newcommand\di{\mathrm{div}}
	\newcommand\Pro{\mathrm{Proj\,}}
	\newcommand\sHom{\mathscr{H}om}
	\newcommand\Supp{\mathrm{Supp\,}}
	\newcommand\Ann{\mathrm{Ann\,}}
	\newcommand\Ext{\mathrm{Ext}}
	\newcommand\sExt{\mathscr{E}xt}
	\newcommand\cdim{\mathrm{codim}}
	\newcommand\Cl{\mathrm{Cl\,}}
	\newcommand\Pic{\mathrm{Pic\,}}
	\newcommand\ord{\mathrm{ord}}
	\newcommand\Gal{\mathrm{Gal}}
	\newcommand\dra{\dashrightarrow}
	\newcommand\reldeg{\mathrm{reldeg}}
	\newcommand{\mib}[1]{\textit{\textbf{#1}}}
	\newcommand\bqed{\hfill $\blacksquare$}

	\numberwithin{equation}{subsection}
	
	\tableofcontents

	\section{Introduction}\label{Sec_intro}
	In the study of Diophantine geometry, a central object of interest is the torsion subgroup of an abelian variety. The uniform boundedness conjecture predicts the existence of a uniform upper bound on the order of torsion subgroups based only on the dimension of the abelian variety and the degree of the base field.
	
	\begin{conjecture}\label{conj:ubc}
		Fix positive integers $g$ and $D$. There exists a positive integer $N(g,D)$ such that for any abelian variety $A$ of dimension $g$ defined over a number field $K$ of degree $[K:\mb{Q}] = D$, the number of $K$-rational torsion points is bounded by $N(g,D)$:
		\[ \# A(K)_{\mr{tors}} \leq N(g,D). \]
	\end{conjecture}
	
	In the case where $A$ is an elliptic curve over $\mb{Q}$, Conjecture~\ref{conj:ubc} was proved by Mazur \cite{Maz77,Maz78} through a classification of torsion subgroups. For general number fields, the conjecture for elliptic curves was advanced by Kamienny \cite{Kam92} and Kamienny--Mazur \cite{KM95}, and was finally completed by Merel \cite{Mer96}.
	
	The higher-dimensional case remains widely open. Given that the moduli space of elliptic curves is one-dimensional, it is natural to investigate abelian schemes over curves. In \cite{CT12}, Cadoret and Tamagawa proved the following result, which can be viewed as a weaker variant of the uniform boundedness conjecture.
	
	\begin{theorem}[{\cite[Theorem 1.1]{CT12}}]\label{thm:CT}
		Let $k$ be a finitely generated field over $\mb{Q}$, $B$ be a smooth geometrically connected curve over $k$, and $\pi: \mc{A} \to B$ be an abelian scheme. For any finite extension $k'/k$ and any prime $\ell$, there exists an integer $N \colonequals  N(\mc{A}, B, k, k', \ell)$ such that for all $b \in B(k')$, 
		\[ \#\mc{A}_b(k')[\ell^{\infty}] \leq N, \]
		where $\mc{A}_b(k')[\ell^{\infty}] \colonequals  \bigcup\limits_{n=0}^{\infty} \mc{A}_b(k')[\ell^n]$ denotes the set of $\ell$-primary torsion points in $\mc{A}_b(k')$. 
	\end{theorem}
	
	\begin{remark}
		By the Mordell's Conjecture (Faltings' Theorem), if the genus $g(B) \geq 2$, the curve $B(k')$ contains only finitely many rational points. Thus $\mc{A}_b(k')\neq \varnothing$ for only finitely many fibers. Consequently, Theorem~\ref{thm:CT} is interesting only when $g(B) = 0$ or $1$.
	\end{remark}
	
	Cadoret and Tamagawa strengthened this result in \cite{CT13}
	by showing that $N$ can be chosen to depend only on
	$\mathcal A$, $B$, $k$, $[k':k]$, and $\ell$. More precisely, for every integer $d\geq1$ and every prime
	$\ell$, there exists a constant $N=N(\mathcal A,B,k,d,\ell)$
	such that
	\[
	\#\mathcal A_b(k')[\ell^\infty]\leq N
	\]
	for every finite extension $k'/k$ with $[k':k]\leq d$
	and every $b\in B(k')$.
	A key geometric ingredient is the growth of the gonality
	of the associated $\ell$-power torsion covers.
	
	The proofs of Cadoret and Tamagawa use Galois representations.
	In this article, we give a different approach based on
	Betti foliations and arithmetic equidistribution,
	which yields a new proof of Theorem~\ref{thm:CT}.
	In Appendix~\ref{sec:appendix}, we also reprove the strong uniform boundedness of $\ell$-primary torsion points.

	\subsection{Main Results}
	Let $k$ be a field of characteristic zero and $B$ be a smooth geometrically connected curve over $k$. For an abelian variety $A$ over $K=k(B)$, the Chow
	$(\overline K/\overline k)$-trace $A_0=\mr{Tr}_{\overline K/\overline k}(A_{\overline K})$ is an abelian variety over $\overline k$, together with
	a $\overline K$-homomorphism $\mr{tr}:A_{0,\overline K}\to A_{\overline K}$ satisfying the following property: for every abelian variety $A'_0$ over
	$\overline k$, every $\overline K$-homomorphism
	$A'_{0,\overline K}\to A_{\overline K}$ factors through
	$\mr{tr}$ via a unique $\overline k$-homomorphism
	$A'_0\to A_0$. The pair $(A_0,\mr{tr})$ is unique up to unique
	isomorphism.
	
	We say that $A$ is \emph{isotrivial} if $\mr{tr}:A_0\to A_{\ov{K}}$
	is an isomorphism. Thus $A$ contains no nontrivial
	isotrivial abelian subvariety if and only if $A_0=0$.
	
	Let $\mathcal A\to B$ be an abelian scheme with
	generic fiber $A$. We say that $\mathcal A\to B$ is \emph{isotrivial}
	if, after a finite surjective base
	change $B'\to B_{\overline k}$, it becomes isomorphic
	to $A_0\times_{\overline k}B'$ for some abelian variety $A_0$ over $\overline k$. Then $\mathcal A\to B$ is isotrivial if and only if
	$A$ is isotrivial. Furthermore, $\mathcal A$ contains no nontrivial
	isotrivial abelian subscheme if and only if $\mr{Tr}_{\overline K/\overline k}(A_{\overline K})=0$.
	
	To provide a geometric proof of Theorem~\ref{thm:CT}, the primary ingredient is the following result:
	
	\begin{theorem}\label{thm:main1}
		Let $k$ be a finitely generated field over $\mb{Q}$, $B$ be a smooth geometrically connected curve over $k$, and $\pi: \mc{A} \to B$ be an abelian scheme. Let $K \colonequals  k(B)$ be the function field of $B$ and $A \colonequals  \mc{A}_K$ be the generic fiber. Suppose $(x_n)_{n \geq 1}$ is a generic and small sequence in $A(\ov{K})$, and let $C_n \colonequals  \Zar{\{x_n\}}$ be the Zariski closure of $x_n$ in $\mc{A}$. If the genus sequence $\big(g(C_n)\big)_{n\geq 1}$ satisfies 
		\[\liminf_{n\to\infty}\frac{g(C_n)}{\deg(x_n)}=0,\] then the abelian scheme $\pi: \mc{A} \to B$ is isotrivial. 
	\end{theorem}
	
	\begin{remark}
		By the Northcott property, for a sequence of distinct points $(x_n)_{n \geq 1}$ with bounded height, the degree $\deg(x_n)$ must tend to infinity as $n\to \infty$. Consequently, Theorem~\ref{thm:main1} implies that if the abelian scheme $\pi: \mathcal{A} \to B$ is non-isotrivial, then for any generic and small sequence $(x_n)_{n \geq 1}$ in $A(\ov{K})$, the genus satisfies $g(C_n)\to \infty$ as $n\to \infty$. 
	\end{remark}
	
	By combining Theorem~\ref{thm:main1} with the Mordell's conjecture and Bogomolov's conjecture, we prove the following result:
	
	\begin{theorem}\label{thm:sec1}
		Let $k$ be a finitely generated field over $\mb{Q}$, $B$ be a smooth curve over $k$, and $\pi: \mc{A} \to B$ be an abelian scheme. Assume $\mc{A}$ contains no non-trivial isotrivial abelian subscheme. Let $P: B \to \mc{A}$ be a section. For any finite extension $k'/k$ and any prime $\ell$, there exists an integer $N \colonequals  N(\mc{A}, P, k, k', \ell)$ such that for any $b \in B(k')$, 
		\[ \#\{z \in \mc{A}_b(k') : \ell^n z = P_b \text{ for some } n \geq 0 \} \leq N. \]
	\end{theorem}
	
	Note that if $P$ is a non-torsion section, the assumption that $A$ contains no non-trivial isotrivial abelian subvariety, is essential.
	
	\begin{example}\label{exam:nontors}
		Let $E$ be an elliptic curve over $k$ with $\operatorname{rank} E(k) \geq 1$. Consider the abelian scheme $\mc{A} \colonequals  E \times_k E$ and let $\pi: \mc{A} \to E$ be the projection onto the second factor. Let $P: E \to \mc{A}$ be the diagonal section. 
		
		Since $\mr{rank\,}E(k)\geq 1$, for any $N > 0$, we can find a point $b \in E(k)$ of the form $b = \ell^N \cdot b_0$ for some non-torsion element $b_0 \in E(k)$. There are at least $N$ distinct points $z \in E(k)$ such that $\ell^n z = b$ for some $n \geq 0$. Thus, the set in Theorem~\ref{thm:sec1} cannot be bounded uniformly for all $b\in E(k)$.
	\end{example}
	
	If $P$ is a torsion section, we may reduce to the case where $P$ is the zero section. Due to the uniform boundedness of torsion points on fibers of isotrivial abelian schemes, we can prove Theorem~\ref{thm:sec1} without the non-isotrivial subscheme assumption. This yields a new proof of Theorem~\ref{thm:CT}.
	
	In Appendix~\ref{sec:appendix}, we strengthen Theorem~\ref{thm:sec1} by showing that the constant $N$ can be chosen to depend only on $\mathcal A/B$, $P$, $k$,
	$[k':k]$, and $\ell$ (see Theorem~\ref{thm:sec_strong}).
	The bounded-degree $\ell$-primary torsion result of
	\cite{CT13} then follows by taking $P=0$ and treating
	the isotrivial part (see Theorem~\ref{thm:CT_alg}).
	
	Our Theorem~\ref{thm:main1} also implies a conjecture proposed by Cadoret and Tamagawa in \cite{CT11}.  For a positive integer $n$, let $A[n]^{\times}$ denote the set of torsion points in $A(\ov{K})$ of exact order $n$. For each $x\in A(\overline K)$, let $C_x^\circ$ be the smooth projective curve over $\overline k$ with function field $\overline kK(x)$.
	
	\begin{conjecture}\label{conj:CT}
		Let $k$ be a finitely generated field over $\mb{Q}$, $B$ be a smooth geometrically connected curve over $k$, and $A$ be an abelian variety over $K:=k(B)$. Assume $A$ contains no non-trivial isotrivial abelian subvariety. Let $g(n) \colonequals  \min\limits_{x \in A[n]^{\times}} g(C_{x}^{\circ})$. Then 
		\[ \lim_{n \to \infty} g(n) = +\infty. \]
	\end{conjecture}
	
	In \cite{CT12}, Cadoret and Tamagawa proved that
	$g(\ell^n)\to\infty$ as $n\to\infty$ for every prime $\ell$.
	In \cite{CT11}, they also proved the conjecture under
	the additional assumption that either $g(B)\geq 1$ or $\mathcal A\to B$ has semistable reduction at all but at most one point of $\overline B\setminus B$.
	In \cite{CT19}, they proved prime-level genus growth
	for abstract modular curves in arbitrary characteristic. Using our geometric approach, we prove Conjecture~\ref{conj:CT} without these additional assumptions (see Corollary~\ref{cor:infgenus}).
	
	There are also related results on gonality. Set
	\[
	\operatorname{gon}(n)\colonequals
	\min_{x\in A[n]^\times}\operatorname{gon}(C_x^\circ).
	\]
	When $k$ is a number field, the results of Ellenberg,
	Hall, and Kowalski in \cite[Sections~2 and~5]{EHK12}
	imply that $\operatorname{gon}(\ell)\to\infty$
	as the prime $\ell$ tends to infinity. Their method uses expansion properties of
	Cayley--Schreier graphs. Cadoret and Tamagawa proved in \cite{CT13} that
	$\operatorname{gon}(\ell^n)\to\infty$ as $n\to\infty$ for every fixed prime $\ell$.
	They also proved gonality growth for prime-level abstract
	modular curves in characteristic zero in \cite{CT19}.
	See also \cite{CT16} for results in positive characteristic. In Appendix~\ref{sec:appendix}, we give another proof that
	$\operatorname{gon}(\ell^n)\to\infty$ as $n\to\infty$
	for every fixed prime $\ell$, using our geometric approach
	(see Corollary~\ref{cor:infgon}).

	\subsection{Ideas of proofs}
	The majority of this article is devoted to the proof of Theorem~\ref{thm:main1}. Here we describe our ideas to prove Theorem~\ref{thm:main1}, which mainly uses the theory of Betti foliations and the arithmetic equidistribution theorem.
	
	For simplicity, let us assume here that $k$ is a number field. For any $b \in B(\mathbb{C})$, the Betti foliation induces a monodromy representation
	\[
	\rho_b: \pi_1(B(\mathbb{C}), b) \to \mr{Aut}\big(H_1(\mc{A}(\mb{C})_b,\mb{Z})\big)\simeq \mathrm{GL}_{2g}(\mathbb{Z}).
	\]
	For a loop $\gamma$ based at $b$, the element $\rho_b(\gamma)$ acts as an automorphism on the real Lie group $\mathcal{A}(\mathbb{C})_b$. Let $\overline{B}$ be the smooth compactification of $B$, and let $s \in \overline{B} \setminus B$ be a boundary point. By choosing a base point $b$ near $s$ and a simple loop $\gamma \in \pi_1(B(\mathbb{C}), b)$ around $s$, the monodromy $\rho_b(\gamma)$ characterizes the reduction of the abelian scheme. 
	
	The proof has three main steps.
	
	\medskip
	\noindent\textbf{Step 1.} Choose pairwise disjoint coordinate disks
	$U_i\simeq D(0,3)$ centered at $s_i$, and set
	\[
	E_i\colonequals\{1\leq |z|\leq2\}\subset U_i.
	\]
	For $b\in E_i$, let $\gamma_b$ be the circle centered
	at $s_i$ and based at $b$. For almost every $b$, this loop has $d_n$ lifts to $C_n$.
	Let $\ell_n(b)$ be their average length with respect
	to the Betti form. We use a length--area estimate and the integral formula
	for the geometric N\'eron--Tate height to obtain a bound
	\[
	\int_{E_i}\ell_n(b)\,\mr{d}\mu_{E_i}(b)
	\leq c\sqrt{\widehat h_L^{\overline H}(x_n)},
	\]
	where $c$ is independent of $n$. Smallness then implies that these average lengths
	tend to zero.
	
	\medskip
	
	\noindent\textbf{Step 2.} Let $r_n(b)$ be the number of connected components
	of $\pi_n^{-1}(\gamma_b)$. Using the arithmetic equidistribution theorem of
	Chen--Moriwaki \cite[Theorem F]{CM24}, we prove
	\[
	\limsup_{n\to\infty}
	\int_{E_i}\frac{r_n(b)}{d_n}\,\mr{d}\mu_{E_i}(b)
	\leq\frac{1}{\ord(s_i)}.
	\]
	Indeed, fix an integer $1\leq q<\ord(s_i)$. Since the average length tends to zero, most points
	in components with at most $q$ lifts are nearly fixed by $\rho_b(\gamma_b^j)$ for some $1\leq j\leq q$. However, for each $b$, the set $\{\rho_b(\gamma_b^j)(z)=z\}$ has Haar measure zero.
	By equidistribution, the average proportion of points in these components tends to zero. Each of the other components contains at least $q+1$ lifts, so $\limsup\limits_{n\to\infty}
	\int_{E_i}\dfrac{r_n(b)}{d_n}\leq \dfrac{1}{q+1}$. Taking $q=\ord(s_i)-1$ when $2\leq\ord(s_i)<\infty$, or letting $q\to\infty$ when $\ord(s_i)=\infty$, gives the desired inequality.
	
	\medskip
	
	\noindent\textbf{Step 3.} We use a counting argument that derived from the Riemann--Hurwitz formula to bound the genus $g(\ov{B})$ and $\# \ov{B}\setminus B$. Applying Step~2, we obtain
	\[
	2g(\overline{B}) - 2 + \#(\overline{B} \setminus B)-\sum_{s \in \overline{B} \setminus B} \frac{1}{\operatorname{ord}(s)} \leq \liminf_{n\to\infty}\frac{2g(C_n)}{d_n}=0.
	\]
	Using the inequality above, we explicitly determine the possible values for the tuple $(\operatorname{ord}(s))_{s \in \overline{B} \setminus B}$. By taking a suitable base change, we reduce the problem to cases where $B$ is isomorphic to $\mathbb{A}^1, \mathbb{G}_m, \mathbb{P}^1$, or an elliptic curve. A hyperbolicity argument then allows us to prove that the abelian scheme $\pi: \mathcal{A} \to B$ is isotrivial.
	
	For a general finitely generated field $k$, we spread
	out the family over a model of $k$ and use the same argument.
	
	\subsection{Notations and Terminology}
	By a \emph{variety}, we mean a separated, integral scheme of finite type over a base field. A \emph{curve} is defined as a variety of dimension one.
	
	By a \emph{line bundle} on a scheme, we mean an invertible sheaf. For any line bundles $L, M$ and integers $a, b$, the notation $aL - bM$ denotes the tensor product $L^{\otimes a} \otimes M^{\otimes (-b)}$.  For the theory of \emph{adelic line bundles}, we refer to \cite{YZ26}.
	
	An \emph{abelian scheme} $\mc{A}$ over a scheme $B$ is a proper, smooth group scheme over $B$ whose geometric fibers are connected. 
	
	For a variety $X$ over $\mb{Q}$, we denote $X(\mb{C})$ the set of $\mb{C}$-points of $X\otimes_{\mb{Q}}\mb{C}$, which admits the structure of a complex analytic space. If $X$ is moreover smooth, then $X(\mb{C})$ is a complex manifold. In particular, if $X$
	is a smooth projective curve, then $X(\mb{C})$ is a finite disjoint union of compact connected Riemann surfaces.
	
	For a curve $C$ over a field $k$ of characteristic zero, denote $\widetilde C$ the smooth projective model of its normalization. We define the \emph{genus} of $C$ by
	\[
	g(C)\colonequals\dim_k H^1(\widetilde C,\mathcal O_{\widetilde C}).
	\]
	Equivalently, if $C_1,\ldots,C_r$ are the connected components
	of $\widetilde C_{\overline k}$, then
	\[
	g(C)=\sum_{i=1}^r g(C_i).
	\]
	
	For a compact connected Riemann surface $C$ with boundary, we define $g(C)$ to be the genus of the closed surface obtained by attaching a disk to each boundary component. If $C$ is disconnected, we define $g(C)$ as the sum of the genera of its connected components.
	
	For a point $x \in \mb{C}$ and a positive real number $r$, we denote the open disk of radius $r$ centered at $x$ by
	\[
	D(x,r) \colonequals \{z \in \mb{C} : |z-x| < r\} \subset \mb{C}.
	\]
	For a point $x = (x_1, \dots, x_n) \in \mb{C}^n$, we denote the polydisk of radius $r$ by
	\[
	\mb{D}(x,r) \colonequals D(x_1, r) \times \cdots \times D(x_n, r) \subset \mb{C}^n.
	\]

	\subsection*{Acknowledgements}
	We would like to thank Xinyi Yuan and Wenbin Luo for some useful discussions. Zhuchao Ji is supported by National Key R\&D Program of China (No.2025YFA1018300), NSFC Grant (No.12401106), and ZPNSF grant (No.XHD24A0201). Junyi Xie is supported by NSFC Grant (No.12271007) and by the Xplorer Prize from the New Cornerstone Science Foundation.
	
	\subsection*{AI usage} All statements and proofs outside Appendix~\ref{sec:appendix} were developed by the authors. In January 2026, the authors posted the paper, excluding Appendix~\ref{sec:appendix}, on arXiv. Recently, the authors used GPT-6 Astra to assist in deriving the strong uniform boundedness of $\ell$-primary torsion points from the main theorem, using an argument similar to that in \cite{CT13}. These results are presented in Appendix~\ref{sec:appendix}. The authors wrote the proofs in that appendix and independently verified the statements and every step of the arguments.

	\section{Preliminaries}\label{Sec_pre}
	\subsection{The moduli space of abelian varieties}
	Let $D=\mr{diag}(d_1,d_2,\dots, d_g)\in M_{g\times g}(\mb{Z})$ be a diagonal matrix with $d_1,d_2,\dots,d_g$ positive integers such that $d_1\mid d_2\mid \dots \mid d_g$. For an integer $N\geq 3$, let $\mb{A}_{g, D}(N)$ be the moduli space of abelian varieties of dimension $g$ with polarization type $D$ and level-$N$-structure. For simplicity, we denote it by $\mb{A}_g$ without ambiguity. More precisely, consider the Siegel upper half space defined by
	\[\mathfrak{H}_g^+\colonequals \{Z=X+\sqrt{-1} Y\in M_{g\times g}(\mb{C})\mid Z=Z^{\top}, Y\text{ is positive definite}\}.\]
	Then this gives a uniformization $u\colon \mathfrak{H}_g^+\ra \mb{A}_g$. It is known that $\mb{A}_g(\mb{C})$ can be identified with the quotient $\Gamma\backslash \mathfrak{H}_g^+$ for a suitable congruence subgroup $\Gamma $ of $\mr{Sp}_{g}(\mb{Z})$. 
	
	We also know that $\mb{A}_g$ is a fine moduli space, which admits a universal family $\pi\colon\mathfrak{A}_g\ra \mb{A}_{g}$. 
	
	Let $B$ be is a smooth quasi-projective variety over $\mathbb{C}$ and $\pi_B\colon\mc{A}\ra B$ be an abelian scheme of relative dimension $g$, where $g\geq 1$ is an integer. Up to taking a finite \'etale covering of $B^{\prime}\ra B$ and the base change $\mc{A}^{\prime}=\mc{A}\times_B B^{\prime}$, there is a modular map  
	\[\begin{tikzcd}
		\mc{A} \arrow[r] \arrow[d, "\pi_B"'] & \mathfrak{A}_g \arrow[d, "\pi"] \\
		B \arrow[r, "\iota_B"]               &\mb{A}_{g}                            
	\end{tikzcd}.\]
	
	For some particular base $B$, the abelian scheme must be isotrivial. 
	
	\begin{proposition}\label{prop:isotriv}
		If $B=\mb{P}^1,\mb{A}^1, \mb{G}_m$ or an elliptic curve, then the abelian scheme $\mc{A}\ra B$ is isotrivial.
	\end{proposition}
	\begin{proof}
		When $B=\mb{A}^1, \mb{G}_m$ or an elliptic curve, there is a holomorphic universal covering $p\colon \mb{C}\ra B(\mb{C})$. The modular map $\iota_B\colon B\ra \mb{A}_{g}$ lifts to a holomorphic map $\wt{\iota_B}\colon\mb{C}\ra \mathfrak{H}_g^+$. Since the Siegel upper half space is biholomorphic to a bounded domain, the map $\wt{\iota_B}$ is constant, which implies the abelian scheme $\mc{A}\ra B$ is isotrivial.
		
		When $B=\mb{P}^1$, for any $s\in B$, $U\colonequals B\setminus\{s\}\simeq \mb{A}^1$. So the abelian scheme $\mc{A}_U\colonequals \mc{A}\times_BU\to U$ is isotrivial. Thus $\mc{A}\to B$ is isotrivial.
	\end{proof}
	
	\subsection{The Betti foliation}
	We assume $B$ is a smooth quasi-projective variety over $\mb{C}$ and $\pi\colon\mathcal{A}\ra B$ is a holomorphic family of abelian varieties over $B$, which is a smooth holomorphic map endowed with a holomorphic section $e\colon B\ra \mathcal{A}$. Every fiber of $\pi$ is an abelian variety with the identity point induced by $e$. 
	
	Let $b\in B$ be a point and $U$ be a connected and simply connected open neighborhood of $b$ in $B$. Denote $\pi_U: \mathcal{A}_U\colonequals\mathcal{A}\times_BU\ra U$ the base change of $\pi$. The \emph{Betti map} \[\beta=\beta_{b,U}\colon \mathcal{A}_U\ra \mathcal{A}_b\]
	is a canonical real analytic map satisfying the following properties:
	\begin{enumerate}[(1)]
		\item  The composition $\mathcal{A}_b \hookrightarrow \mathcal{A}_U \xrightarrow{\beta} \mathcal{A}_b$ is the identity map.
		
		\item  For any point $b^{\prime} \in U$, the composition $\mathcal{A}_{b^{\prime}} \hookrightarrow \mathcal{A}_U \xrightarrow{\beta} \mathcal{A}_b$ is an isomorphism of real Lie groups.
		
		\item  The induced map
		
		$$
		\wt{\beta}=(\beta, \pi)\colon \mathcal{A}_U \longrightarrow \mathcal{A}_b \times U
		$$
		
		is a real analytic diffeomorphism of manifolds.
		
		\item For any $x \in \mathcal{A}_U$, the fiber $\beta^{-1}(\beta(x))$ is a complex analytic subset of $\mathcal{A}_U$, which is biholomorphic to $U$.
		
	\end{enumerate}
	
	The Betti map can be constructed as follows. Fix a basis of $H_1(\mc{A}_b,\mb{Z})$. Since $U$ is simply-connected, we can extend it by continuity to all fibers above $U$. Then we get canonical isomorphisms $H_1(\mc{A}_{b^{\prime}},\mb{Z})\ra H_1(\mc{A}_b,\mb{Z})$ for all $b^{\prime}\in U$. The isomorphism in (2) can be defined by the composition
	\[\mathcal{A}_{b^{\prime}} \longrightarrow H_1\left(\mathcal{A}_{b^{\prime}}, \mathbb{R}\right) / H_1\left(\mathcal{A}_{b^{\prime}}, \mathbb{Z}\right) \longrightarrow H_1\left(\mathcal{A}_b, \mathbb{R}\right) / H_1\left(\mathcal{A}_b, \mathbb{Z}\right) \longrightarrow \mathcal{A}_b,\]
	which determines the Betti map $\beta_{b,U}$. 
	
	For a point $x\in \mc{A}_U$, the fiber $\mc{F}_{x,U}\colonequals \beta_{b,U}^{-1}\left(\beta_{b,U}(x)\right)$ is called the \emph{local Betti leaf} at $x$ over $U$. It is independent of the choice of $b\in U$, and the germ $\mc{F}_x$ is independent of the choice of $U$. A \emph{Betti leaf} is a connected subset $\mc{F}_0$ of $\mc{A}$ whose germ at $x\in \mc{F}_0$ is $\mc{F}_x$. This forms a \emph{Betti foliation}.
	
	We have the following basic properties for the Betti foliation.
	
	\begin{itemize}
		\item In general, the Betti map $\beta_{b, U}$ is not holomorphic, but the local Betti leaf $\mc{F}_{x, U}$ is a complex submanifold of $\mc{A}_U$.
		\item The Betti leaf $\mc{F}$ is transverse to the fibers of $\pi: \mc{A}\ra B$ and the map $\pi|_{\mc{F}_0}\colon \mc{F}_0\to B$ is a holomorphic covering map. 
		\item Any connected component of a torsion multi-section of $\pi\colon \mc{A}\to B$ is a Betti leaf.
	\end{itemize}
	
	Let $\gamma\colon [0,1]\ra B$ be a loop in $B$ with base point $b$, moves along $\gamma(t)$ from $t=0$ to $t=1$, a basis of $H_1(\mc{A}_b, \mb{Z})$ transform to a second basis of $H_1(\mc{A}_b, \mb{Z})$, which gives an element of $\mr{Aut}\left(H_1(\mc{A}_b, \mb{Z})\right)\simeq\mr{GL}_{2g}(\mb{Z})$. This induces a monodromy 
	\[\rho_b: \pi_1(B,b)\ra \mr{GL}_{2g}(\mb{Z}).\]
	Note that the monodromy $\rho_b(\gamma)$ also gives a linear transfromation of $H_1(\mc{A}_b,\mb{R})$ which preserve the lattice $H_1(\mc{A}_b, \mb{Z})$, so it gives a linear diffeomorphism of $H_1(\mc{A}_b,\mb{R})/H_1(\mc{A}_b, \mb{Z}) \cra \mc{A}_b$. There is another way to describe this diffeomorphism. Let $x\in \mc{A}_b$ be a point and $\mc{F}_0$ be the Betti leaf containing $x$. The loop $\gamma$ lifts to a unique path
	\[\wt{\gamma}_x\colon [0,1]\to \mc{F}_0\]
	such that $\pi\circ \wt{\gamma}_x=\gamma$ and $\wt{\gamma}_x(0)=x$. Then the map
	\[\mc{A}_b\ra \mc{A}_b,\quad x\mapsto \wt{\gamma}_x(1)\]
	is just the linear diffeomorphism of $\mc{A}_b$ induced by $\rho_b(\gamma)$. 
	
	We recall the archimedean N\'eron--Ogg--Shafarevich criterion
	(see, for example, \cite[Proposition~2.1]{Oh23}).
	
	\begin{proposition}\label{prop:monotriv}
		Let $\pi:\mathcal A\to B$ be a holomorphic family of abelian varieties over a smooth
		complex algebraic curve, and let $\overline B$ be the smooth
		compactification of $B$. Fix $s\in\overline B\setminus B$ and choose a sufficiently small
		analytic disk $D\subset\overline B$ centered at $s$, so that
		$D^{\times}\colonequals D\setminus\{s\}\subset B$.
		Then $\mathcal A\to B$
		extends to a holomorphic family of abelian varieties over $B\cup\{s\}$ if and only if
		the local monodromy representation
		\[
		\rho_b:\pi_1(D^{\times},b)\longrightarrow
		\operatorname{Aut}\bigl(H_1(\mathcal A_b,\mathbb Z)\bigr)
		\]
		is trivial, where $b\in D^{\times}$ is a fixed base point.    
	\end{proposition}

	\subsubsection*{The Betti form}
	Here, we review the theory of Betti forms. These notions were originally introduced by Mok \cite{Mok91} to study the Mordell--Weil groups of abelian varieties over complex function fields.
	
	Let $\mathcal{L}$ be a line bundle on $\mathcal{A}$, for any $b\in B$, there exists a unique smooth $(1,1)$-form $\omega_b\in c_1(\mathcal{L}|_{\mc{A}_b})$ on $\mathcal{A}_b$ which is invariant under translations. The \emph{Betti form} $\omega=\omega(\mathcal{L})$ can be constructed as follows.
	
	Let $U$ be a connected and simply connected open neighborhood of $b$ and $\beta_{b, U}$ be a Betti map. We define $\omega_U\colonequals \beta_{b,U}^*\omega_b$, which is a closed $(1,1)$-form on $\mathcal{A}_U$ and independent of the choice of $b\in U$ for fixed $U$. By gluing $\omega_U$ for different $U$, we get the desired Betti form $\omega$. This form is uniquely determined by the generic fiber $\mathcal{L}_{\eta}$ on $\mc{A}_{\eta}$, where $\eta$ denotes the generic point of $B$.
	
	The Betti form $\omega$ is a real analytic closed $(1,1)$-form satisfying $[m]^*\omega=m^2\omega$ for any integer $m$. For a point $x\in \mathcal{A}$ with $\pi(x)=b$, let $\mr{pr}_x: T_x\mathcal{A}\to T_x\mathcal{A}_{b}$ be the projection to the first factor in \[T_x\mathcal{A}=T_x\mathcal{A}_{b}\oplus T_x\mathcal{\mathcal{F}},\]
	where $\mathcal{F}$ is the Betti leaf at $x$. By the construction of the Betti form, we have
	\[\omega(v_1,v_2)=\omega_{b}(\mr{pr}_x(v_1),\mr{pr}_x(v_2))\]
	for any $v_1, v_2\in T_x\mathcal{A}$. We write
	\[
	\|v\|_b
	\colonequals
	\sqrt{\omega_b(v,Jv)}
	\]
	for a vertical tangent vector $v\in T_x\mathcal A_b$ when $\omega_b$ is positive.

	From now on, assume that $\mathcal L$ is ample on every fiber $\mc{A}_b$. Then $\omega_b$ is positive and $\omega$ is semipositive. The form $\omega_b$ induces a distance $\operatorname{dist}_b$ on $\mathcal A_b$.
	
	\begin{definition}
		Suppose $X$ is a complex manifold and $\Omega$ is a smooth semipositive $(1,1)$-form on $X$. For a piecewise smooth path $\gamma:[a,b]\to X$, the \emph{length} of $\gamma$ with respect to $\Omega$ is defined by
		\[L_{\Omega}(\gamma)\colonequals \int_{a}^b\sqrt{\Omega\left(\gamma^{\prime}(t),J\gamma^{\prime}(t)\right)}\mr{d}t\geq 0.\]
	\end{definition}
	
	The following lemma will be used in the proof of
	Proposition~\ref{prop:fg_cycles}.
	\begin{lemma}\label{lem:dist}
		Fix a smooth Riemannian metric $h$ on $B$.
		Let $E\subset B$ be compact and let $R>0$.
		There exists a constant $c>0$ such that, for every piecewise
		smooth loop $\gamma$ in $E$ based at $b$ with
		$L_{h}(\gamma)\leq R$, and every closed lift
		$\widetilde\gamma$ of $\gamma$ starting at $x$, we have
		\[
		\operatorname{dist}_b\bigl(x,\rho_b(\gamma)(x)\bigr)
		\leq cL_\omega(\widetilde\gamma).
		\]
	\end{lemma}
	\begin{proof}
		For $t\in[0,1]$, denote by
		\[
		\beta_{\gamma(t)}^b:
		\mathcal A_{\gamma(t)}
		\longrightarrow\mathcal A_b
		\]
		the Betti transport along $\gamma|_{[t,1]}$.
		We claim that there exists a constant $c>0$, depending only
		on $E$ and $R$, such that
		\[
		\|(\beta_{\gamma(t)}^b)_*v\|_b
		\leq c\|v\|_{\gamma(t)}
		\]
		for every $t\in[0,1]$, every
		$y\in\mathcal A_{\gamma(t)}$, and every
		$v\in T_y\mathcal A_{\gamma(t)}$.
		
		In fact, since $E$ is compact, the covariant derivative of the metric
		$\omega_b(\,\cdot\,,J\,\cdot\,)$ with respect to the Betti
		connection is bounded on $E$. Thus there exists $a\geq0$, depending only on $E$,
		such that every Betti-parallel vector field $v(s)$ along
		$\gamma$ satisfies
		\[
		\left|
		\frac{\mathrm d}{\mathrm ds}\|v(s)\|_{\gamma(s)}^2
		\right|
		\leq
		2a|\gamma'(s)|_{h}\|v(s)\|_{\gamma(s)}^2.
		\]
		By Gronwall's inequality \cite[Lemma 2.7]{Tes12},
		\[
		\|(\beta_{\gamma(t)}^b)_*v\|_b^2
		\leq
		\exp\left(
		2a\int_t^1|\gamma'(s)|_{h}\,\mathrm ds
		\right)\|v\|_{\gamma(t)}^2
		\leq e^{2aR}\|v\|_{\gamma(t)}^2.
		\]
		This proves the claim with $c=e^{aR}$.
		
		Now consider the path $\alpha(t)=\beta_{\gamma(t)}^b
		\bigl(\widetilde\gamma(t)\bigr)$ in $\mathcal A_b$.
		Its endpoints are $\alpha(0)=\rho_b(\gamma)(x)$ and
		$\alpha(1)=x$. Let $
		\operatorname{pr}_{\widetilde\gamma(t)}:
		T_{\widetilde\gamma(t)}\mathcal A
		\longrightarrow
		T_{\widetilde\gamma(t)}\mathcal A_{\gamma(t)}$ be the vertical projection along the Betti foliation. Then
		\[
		\alpha'(t)=
		(\beta_{\gamma(t)}^b)_*
		\left(
		\operatorname{pr}_{\widetilde\gamma(t)}
		\bigl(\widetilde\gamma'(t)\bigr)
		\right)
		\]
		By the claim, we obtain
		\[
		\omega_b\bigl(\alpha'(t),J\alpha'(t)\bigr)
		\leq c^2\omega_{\gamma(t)}
		\left(
		\operatorname{pr}_{\widetilde\gamma(t)}
		\bigl(\widetilde\gamma'(t)\bigr),
		\operatorname{pr}_{\widetilde\gamma(t)}
		\bigl(J\widetilde\gamma'(t)\bigr)
		\right)=c^2\omega\bigl(\widetilde\gamma'(t),
		J\widetilde\gamma'(t)\bigr).
		\]
		By integrating, we obtain
		\[
		\operatorname{dist}_b\bigl(x,\rho_b(\gamma)(x)\bigr)
		\leq L_{\omega_b}(\alpha)
		\leq cL_\omega(\widetilde\gamma),
		\]
		which completes the proof.
	\end{proof}
	
	\subsection{The N\'eron--Tate height}\label{Sec:ht}
	Throughout this subsection, we follow the language and conventions of \cite{YZ26}. Let $B$ be a smooth quasi-projective variety of dimension $d\geq1$ over a number field $k$, and let $K\colonequals k(B)$e its function field. Fix a big and nef adelic line bundle $\overline H$ on $B$, which can be viewed as an element in $\widehat{\mr{Pic}}(K/\mb{Z})_{\mr{nef}}$.
	
	Let $X$ be a projective variety over $K$, let $L$ be a line bundle on $X$, and let $\overline L\in\widehat{\operatorname{Pic}}(X/\mathbb Z)_{\mathrm{int},\mathbb Q}$ be an integrable adelic line bundle extending $L$. For $x\in X(\overline K)$, define its \emph{Moriwaki height} by
	\[
	h_{\overline L}^{\overline H}(x)
	\colonequals
	\frac{\langle x^*\overline L\rangle_{K(x)/K}\cdot\overline H^d}
	{[K(x):K]}.
	\]
	Here $x$ is viewed as a morphism $\operatorname{Spec}K(x)\to X$, and $\langle\,\cdot\,\rangle_{K(x)/K}$ is the Deligne pairing in relative dimension zero, see \cite[Sections 4]{YZ26} for details.
	
	We now specialize to an abelian variety $A$ over $K$. Let $L\in\mr{Pic}(A)$ be a symmetric ample line bundle. By \cite[Theorem 6.1.3]{YZ26}, $L$ admits a canonical
	nef adelic extension $\overline L^{\mathrm{can}}\in\widehat{\mr{Pic}}(A/\mb{Z})_{\mr{nef}}$
	satisfying $[2]^*\overline L^{\mathrm{can}}
	\simeq 4\overline L^{\mathrm{can}}$. Indeed, after shrinking $B$ if necessary, let
	$\pi:\mathcal A\to B$ be an abelian scheme extending $A$,
	and let $\mathcal L$ be a symmetric, relatively ample
	line bundle on $\mathcal A$ extending $L$ and equipped
	with a rigidification along the zero section. The Tate limiting process induces a canonical nef adelic
	extension $\overline{\mathcal L}^{\mathrm{can}}$. Viewing it as an element of
	$\widehat{\mr{Pic}}(A/\mb{Z})_{\mr{nef}}$ gives
	$\overline L^{\mathrm{can}}$, which is independent of
	the choice of $(\mathcal A,\mathcal L)$.
	
	The \emph{N\'eron--Tate height} is defined by
	\[
	\widehat h_L^{\overline H}(x)
	\colonequals h_{\overline L^{\mathrm{can}}}^{\overline H}(x)
	=
	\frac{\langle x^*\overline L^{\mathrm{can}}\rangle_{K(x)/K}
		\cdot\overline H^d}{[K(x):K]}.
	\]
	Equivalently, for any integrable adelic extension $\overline L_0$ of $L$, we have
	\[
	\widehat h_L^{\overline H}(x)=
	\lim_{n\to\infty}\frac{1}{4^n}
	h_{\overline L_0}^{\overline H}([2^n]x).
	\]
	The N\'eron--Tate height satisfies the following properties.
	\begin{proposition}[{\cite[Theorem 5.3.1]{YZ26}}]\label{prop:ht}
		Let $L$ be an ample symmetric line bundle on $A$, and let $\overline{H}$ be a big and nef adelic line bundle on $B$. Then:
		\begin{enumerate}
			\renewcommand{\labelenumi}{(\theenumi)}
			\item(The Northcott property) For any real numbers $D,M\in\mathbb{R}$, the set
			\[\{x\in A(\ov{K}): \deg(x)<D,\  \widehat{h}_{L}^{\ov{H}}(x)<M\}\]
			is finite.
			\item For any integer $m$ and $x\in A(\ov{K})$, $\widehat{h}_{L}^{\ov{H}}(mx)=m^2\widehat{h}_{L}^{\ov{H}}(x)$.
			\item A point $x\in A(\ov{K})$ is torsion if and only if $\widehat{h}_{L}^{\ov{H}}(x)=0$.
			\item For any big and nef adelic line bundles $\overline{H},\overline{H}^{\prime}$ on $B$, there exist constants $c_1,c_2>0$ such that
			\[c_1 \widehat{h}_{L}^{\ov{H}}(x)\leq \widehat{h}_{L}^{\ov{H}^{\prime}}(x)\leq c_2 \widehat{h}_{L}^{\ov{H}}(x),\quad \forall x\in A(\ov{K}).\]
		\end{enumerate}
	\end{proposition}
	
	\subsubsection*{The geometric height}
	Let $\wt{L}^{\mr{can}}$ be the geometric part of $\ov{L}^{\mr{can}}$, which is the image of $\ov{L}^{\mr{can}}$ under the canonical composition
	\[\widehat{\mr{Pic}}(A/\mb{Z})_{\mr{nef},\mb{Q}}\longrightarrow \widehat{\mr{Pic}}(A/k)_{\mr{nef},\mb{Q}}.\]
	Similarly, let $\wt{H}\in \widehat{\mr{Pic}}(K/k)_{\mr{nef},\mb{Q}}$ be the geometric part of $\ov{H}$. Define \emph{the geometric N\'eron--Tate height} by
	\[
	\widehat h_{L,\mathrm{geom}}^H(x)
	\colonequals
	\frac{\langle x^*\widetilde L^{\mathrm{can}}\rangle_{K(x)/K}
		\cdot\widetilde H^{d-1}}{[K(x):K]}.
	\]
	Here $\langle\,\cdot\,\rangle_{K(x)/K}$ is the Deligne pairing in relative dimension zero.
	
	The geometric height is controlled by the arithmetic one:
	\begin{proposition}\label{prop:htcomp}
		Let $\overline{H}$ be a big and nef adelic line bundle on $B$. Then there exists a constant $\varepsilon>0$ such that
		\[\varepsilon\cdot\widehat{h}_{L,\mr{geom}}^H(x)\leq \widehat{h}_{L}^{\ov{H}}(x),\quad \forall x\in A(\ov{K}).\]
	\end{proposition}
	\begin{proof}
		Let $\ov{E}\in \widehat{\mr{Pic}}(k/\mb{Z})_{\mr{nef}}$ be an adelic line bundle with $\widehat{\deg}(\ov{E})=\varepsilon$. Denote $\psi:B\to \Spe k$ the structural map. Then by Yuan's theorem \cite[Theorem 5.2.2]{YZ26}, 
		\[\widehat{\mr{vol}}\big(\ov{H}-\psi^*\ov{E}\big)\geq \ov{H}^{d+1}-(d+1)\ov{H}^d\cdot\psi^*\ov{E}=\ov{H}^{d+1}-\varepsilon(d+1)H^d.\]
		For sufficiently small $\varepsilon>0$, the right-hand side is positive, hence $\widehat{\mr{vol}}\big(\ov{H}-\psi^*\ov{E}\big)>0$. Thus there exists an effective section $s$ of a multiple of $\ov{H}-\psi^*\ov{E}$ with $\widehat{\di}(s)\geq 0$. It follows that
		\[\widehat{h}_L^{\ov{H}}(x)=\frac{\langle x^*\ov{L}^{\mr{can}}\rangle_{K(x)/K}\cdot \ov{H}^{d}}{[K(x):K]}\geq\frac{\langle x^*\ov{L}^{\mr{can}}\rangle_{K(x)/K}\cdot \ov{H}^{d-1}\cdot \psi^*\ov{E}}{[K(x):K]}.\]
		Since
		\[\frac{\langle x^*\ov{L}^{\mr{can}}\rangle_{K(x)/K}\cdot \ov{H}^{d-1}\cdot \psi^*\ov{E}}{[K(x):K]}=\frac{\varepsilon\langle x^*\wt{L}^{\mr{can}}\rangle_{K(x)/K}\cdot \wt{H}^{d-1}}{[K(x):K]}=\varepsilon\,\widehat{h}_{L,\mr{geom}}^H(x),\]
		we obtain
		\[\widehat{h}_{L}^{\ov{H}}(x)\geq \varepsilon\cdot\widehat{h}_{L,\mr{geom}}^H(x),\quad \forall x\in A(\ov{K}).\]
	\end{proof}
	
	Fix an embedding $k\hookrightarrow\mathbb C$, and view
	$B$ and $\mathcal A$ as complex varieties via this embedding. Assume that $H$ extends to a smooth projective compactification of $B$, with ample underlying line bundle and smooth positive curvature $c_1(\ov{H})_{\infty}$. The $(1,1)$-form $c_1(\overline{\mathcal L}^{\mathrm{can}})$ on
	$\mathcal A(\mb{C})$ coincides with the Betti form
	$\omega(\mathcal L)$. For $x\in A(\overline K)$, let $\Delta_x$ be its Zariski
	closure in $\mathcal A$. By \cite[Theorem B]{GV25}, the geometric N\'eron--Tate height can be interpreted as an integration:
	\[
	\widehat h_{L,\mathrm{geom}}^H(x)=\frac{1}{\deg(x)}
	\int_{\mathcal A(\mathbb C)}
	[\Delta_x]\wedge c_1(\overline{\mathcal L}^{\mathrm{can}})\wedge\pi^*c_1(\overline H)_\infty^{d-1},\]
	where $[\Delta_x]$ is the current of integration over $\Delta_x(\mathbb C)$.

	\subsection{The equidistribution theorem}\label{Sec_equi}
	Let $B$ be a projective variety over $\mb{Q}$ of dimension $d$, with function field
	$K\colonequals\mb{Q}(B)$. Let $X$ be a projective variety over $K$ of dimension $n$,
	and let $\pi:\mc{X}\to B$ be a flat morphism of projective varieties with generic fiber $X$.
	Fix an ample line bundle $L$ on $X$ and a nef adelic
	extension $\ov{L}\in\widehat{\mr{Pic}}(\mc{X}_U/\mb{Z})_{\mr{nef}}$, where $U\subset B$ is a nonempty open subset and $\mc{X}_U\colonequals \mc{X}\times_BU$.
	
	Let $\ov{H}$ be a big and nef adelic line bundle on $B$ with smooth positive curvature $c_1(\ov{H})_\infty$. The \emph{Moriwaki height} of a closed subvariety
	$Z\subset X$ with respect to $\ov{L}$ and $\ov{H}$ is defined by
	\[h_{\ov{L}}^{\ov{H}}(Z)\colonequals \frac{\big(\ov{L}|_{Z}\big)^{\dim Z+1}\cdot\ov{H}^{d}}{(\dim Z+1)\deg_{L}(Z/K)}.\]
	Here the intersection number is taken in $\widehat{\mr{Pic}}(Z/\mb{Z})_{\mr{int},\mb{Q}}$ and $\deg_{L}(Z/K)\colonequals \big(L|_Z\big)^{\dim Z}$. For an infinite sequence $(x_m)_{m \geq 1}$ in $X(\overline{K})$, we say that $(x_m)_{m \geq 1}$ is $h_{\ov{L}}^{\ov{H}}$-\emph{small} if 
	\[h_{\ov{L}}^{\ov{H}}(x_m) \rightarrow h_{\ov{L}}^{\ov{H}}(X) \text{ as } m \to \infty,\] 
	and is \emph{generic} if any infinite subsequence of $(x_m)_{m \geq 1}$ is Zariski dense in $X$. 
	
	\begin{remark}
		When $\pi: \mc{A}\to B$ is an abelian scheme and $\ov{L}=\ov{L}^{\mr{can}}$ is the canonical adelic line bundle extending an ample and symmetric line bundle $L$, the Moriwaki height $h_{\ov{L}^{\mr{can}}}^{\ov{H}}: A(\ov{K})\to \mb{R}$ is the N\'eron--Tate height. Moreover, one has $h_{\ov{L}^{\mr{can}}}^{\ov{H}}(A)=0$. In this case, a sequence $(x_m)_{m\geq 1}$ is $h_{\ov{L}^{\mr{can}}}^{\ov{H}}$-small if and only if $\lim\limits_{m\to \infty}\widehat{h}_{L}^{\ov{H}}(x_m)=0$. We will simply call it \emph{small} if there is no ambiguity.
	\end{remark}
	
	A point $v\in U(\mb{C})$ is called a \emph{fully transcendental closed point} with respect to $U/\mb{Q}$ if the image of \[\Spe \mb{C}\ora{v} U_{\mb{C}}\to U\] is the generic point of $U$. Such a fully transcendental point $v\in U(\mb{C})$ determines an embedding $\sigma_v:K\hookrightarrow \mb{C}$, and we have a canonical identification $X_{\sigma_v}(\mb{C})=\mc{X}(\mb{C})_{v}$. 
	
	For each point $x\in X(\ov{K})$, let $\Delta_x\subset\mc{X}_U$ be the Zariski closure of its
	image, and let $[\Delta_x]$ be the current on $\mc{X}_U(\mb{C})$ given by integration over
	$\Delta_x(\mb{C})$. Fix a fully transcendental point $v\in U(\mb{C})$ and extend $\sigma_v$ to an embedding $\ov{K}\hookrightarrow\mb{C}$, Through this embedding, the Galois orbit of $x$ can be identified as a subset of $\mc{X}(\mb{C})_{v}$:
	\[
	O_{v}(x) \colonequals \mathrm{Gal}(\overline{K}/K)\cdot x \subset X(\overline{K}) \subset X_{\sigma_v}(\mathbb{C})=\mc{X}(\mb{C})_{v}.
	\]
	This set is independent of the chosen extension of $\sigma_v$. The slice of $[\Delta_{x_m}]$ along the fiber over $v$ is the discrete measure
	\[
	[\Delta_{x_m}]_v=\sum_{z \in O_{v}(x)} \delta_z,
	\]
	where $\delta_z$ denotes the Dirac measure at $z\in X_{\sigma_v}(\mathbb{C})$. 
	
	The following is a special case of the equidistribution theorem of Chen--Moriwaki:
	
	\begin{theorem}[\cite{CM24}]\label{thm:equidis}
		With notations as above. Let $(x_m)_{m\geq 1}$ be a generic and small sequence in $X(\overline{K})$. Then for the archimedean place of $\mb{Q}$, there is a weak convergence
		\[\frac{1}{\deg(x_m)}[\Delta_{x_m}]\wedge c_1(\pi^*\ov{H})^{d}_{\infty}\longrightarrow \frac{1}{\deg_L(X/K)}c_1(\ov{L})^n_{\infty}\wedge \pi^*c_1(\ov{H})^{d}_{\infty}\]
		of measures on $\mc{X}_U(\mb{C})$.
	\end{theorem}
	
	\begin{remark}
		The equidistribution theorem in this relative setting was first proved in \cite[Theorem 6.1]{Mor00} under an additional hypothesis called the \emph{Moriwaki condition}. A similar result for quasi-projective varieties was proved by Yuan--Zhang in \cite[Theorem 5.4.6]{YZ26}. Chen--Moriwaki removed the Moriwaki condition in \cite[Theorem F]{CM24}, proving a more general form of the theorem in the setting of adelic curves. In this article, we only require the above special case, corresponding to a polarized adelic structure in the sense of Chen--Moriwaki. 
	\end{remark}
	
	\subsection{Other lemmas}
	In this subsection, we provide several technical lemmas on Riemann surfaces that are required for the proofs in Section~\ref{Sec:proof}. Readers may skip the proofs in this subsection at a first reading. \medskip

	\noindent The following result is a length-area inequality, which plays a key role in the proof of Lemma~\ref{lem:average_length}.
	\begin{lemma}\label{lem:area}
		Let $E\colonequals\{z\in\mathbb C:1\leq |z|\leq2\}\subset \mb{C}$, and let $\omega=\dfrac{\mathrm i}{2}\lambda(z)\,
		\mathrm dz\wedge\mathrm d\overline z$, where $\lambda$ is a nonnegative Borel measurable function on $E$ such that $\int_E\omega<\infty$. For $r\in[1,2]$, set $\gamma_r\colonequals\partial D(0,r)$.
		Then
		\[
		\int_E\omega
		\geq \frac{1}{4\pi}\int_1^2 L_\omega(\gamma_r)^2\,\mathrm dr.
		\]
	\end{lemma}
	
	\begin{proof}
		By the Cauchy--Schwarz inequality,
		for almost every $r\in[1,2]$ we have
		\[
		L_\omega(\gamma_r)^2
		=
		r^2\left(
		\int_0^{2\pi}\sqrt{\lambda(re^{\mathrm i\theta})}\,
		\mathrm d\theta
		\right)^2
		\leq
		2\pi r^2\int_0^{2\pi}
		\lambda(re^{\mathrm i\theta})\,\mathrm d\theta.
		\]
		Integrating over $r$, we obtain
		\[
		\begin{aligned}
			\int_1^2 L_\omega(\gamma_r)^2\,\mathrm dr
			&\leq
			2\pi\int_1^2\int_0^{2\pi}
			r^2\lambda(re^{\mathrm i\theta})\,
			\mathrm d\theta\,\mathrm dr\\
			&\leq
			4\pi\int_1^2\int_0^{2\pi}
			r\lambda(re^{\mathrm i\theta})\,
			\mathrm d\theta\,\mathrm dr
			=
			4\pi\int_E\omega.
		\end{aligned}
		\]
		This completes the proof.
	\end{proof}

	The following lemma provides an inequality comparing the lengths of a path and its lifts, which will be used in the proof of Lemma~\ref{lem:average_length}.
	\begin{lemma}\label{lem:preimage}
		Let $f:X\to Y$ be a finite holomorphic covering of Riemann surfaces of degree $d$, with $X$ possibly disconnected, and let $\omega$ be a smooth semipositive $(1,1)$-form on $X$.
		For a piecewise smooth path $\gamma:[0,1]\to Y$, let
		$\widetilde{\gamma}_1,\ldots,\widetilde{\gamma}_d: [0,1]\to X$ be distinct lifts of $\gamma$. Then
		\[
		L_{f_*\omega}(\gamma)
		\geq
		\frac{1}{\sqrt d}
		\sum_{j=1}^d L_\omega(\widetilde{\gamma}_j).
		\]
	\end{lemma}
	
	\begin{proof}
		For any $y\in Y$, choose a coordinate neighborhood $U$ such that
		$f^{-1}(U)=\coprod\limits_{j=1}^d V_j$ and each restriction
		$f|_{V_j}:V_j\to U$ is biholomorphic.
		Let $\sigma_j:U\to V_j$ be its inverse. Then
		\[
		(f_*\omega)|_U
		=
		\sum_{j=1}^d\sigma_j^*(\omega|_{V_j}).
		\]
		In a local coordinate $z$ on $U$, write
		\[
		(f_*\omega)|_U
		=
		\frac{\mathrm i}{2}\lambda(z)\,
		\mathrm dz\wedge\mathrm d\overline z,
		\qquad
		\sigma_j^*(\omega|_{V_j})
		=
		\frac{\mathrm i}{2}\lambda_j(z)\,
		\mathrm dz\wedge\mathrm d\overline z,
		\]
		where $\lambda_j\geq0$. Thus $\lambda(z)=\sum\limits_{j=1}^d\lambda_j(z)$.
		
		Assume first that $\gamma([0,1])\subset U$. We have
		$\widetilde{\gamma}_j=\sigma_j\circ\gamma$.
		By the Cauchy--Schwarz inequality,
		\begin{align*}
			L_{f_*\omega}(\gamma)
			&=
			\int_0^1\sqrt{\lambda(\gamma(t))}\,
			|\gamma'(t)|\,\mathrm dt
			=
			\int_0^1
			\sqrt{\sum_{j=1}^d\lambda_j(\gamma(t))}\,
			|\gamma'(t)|\,\mathrm dt\\
			&\geq
			\frac{1}{\sqrt d}
			\sum_{j=1}^d
			\int_0^1\sqrt{\lambda_j(\gamma(t))}\,
			|\gamma'(t)|\,\mathrm dt=
			\frac{1}{\sqrt d}
			\sum_{j=1}^d L_\omega(\widetilde{\gamma}_j).
		\end{align*}
		
		For a general path $\gamma$, choose a partition
		\[
		0=t_0<t_1<\cdots<t_m=1
		\]
		such that $\gamma([t_k,t_{k+1}])\subset U_k$ for suitable
		coordinate neighborhoods $U_k$ as above.
		For $0\leq k<m$, set
		$\gamma^{(k)}\colonequals\gamma|_{[t_k,t_{k+1}]}$ and
		$\widetilde{\gamma}_j^{(k)}
		\colonequals\widetilde{\gamma}_j|_{[t_k,t_{k+1}]}$.
		Applying the preceding argument to each segment gives
		\[
		L_{f_*\omega}(\gamma^{(k)})
		\geq
		\frac{1}{\sqrt d}
		\sum_{j=1}^d
		L_\omega(\widetilde{\gamma}_j^{(k)}).
		\]
		By summing the above inequalities over all $k$, the result follows.
	\end{proof}
	
	The following lemma, which is a corollary of the Riemann--Hurwitz formula, is essential for the proof of Proposition~\ref{prop:fg_kappa}.
	
	\begin{lemma}\label{lem:Riehur}
		Let $p:X\to Y$ be a surjective holomorphic map of degree $d$
		between compact Riemann surfaces.
		Let $D_1,\ldots,D_n\subset Y$ be pairwise disjoint closed disks, none of which contains a
		branched point of $p$ on its boundary. For each $i$, let $r_i$ be the number of connected components
		of $p^{-1}(\partial D_i)$. Then
		\[
		\sum_{i=1}^n r_i
		\geq
		d\bigl(2g(Y)-2+n\bigr)-\bigl(2g(X)-2\bigr).
		\]
	\end{lemma}
	
	\begin{proof}
		Fix $1\leq i\leq n$, and write
		$p^{-1}(D_i)=\coprod\limits_{j=1}^{m_i}\mathcal K_{i,j}$
		for its decomposition into connected components.
		Each $\mathcal K_{i,j}$ is a compact Riemann surface
		with nonempty boundary.
		Let $r_{i,j}$ be the number of connected components of
		$\partial\mathcal K_{i,j}$. So
		$r_i=\sum\limits_{j=1}^{m_i}r_{i,j}$.
		
		We apply the Riemann--Hurwitz formula to $\mathcal{K}_{i,j} \to D_i$. Recall that for a compact connected Riemann surface $C$ with boundary,
		\[
		\chi(C)=2-2g(C)-b(C),
		\]
		where $b(C)$ is the number of boundary components and
		$\chi(C)$ is the Euler characteristic of $C$. Since $D_i$ is a disk, we have $\chi(D_i)=1$. Thus,
		\[
		\begin{aligned}
			\sum_{x\in\mathcal K_{i,j}}(e_x-1)
			&=
			\deg(\mathcal K_{i,j}\to D_i)-\chi(\mathcal K_{i,j})\\
			&=
			\deg(\mathcal K_{i,j}\to D_i)
			-\bigl(2-2g(\mathcal K_{i,j})-r_{i,j}\bigr).
		\end{aligned}
		\]
		Here $e_x$ denote the ramification index of $p$ at $x$. Since $g(\mathcal K_{i,j})\geq0$ and $r_{i,j}\geq1$,
		we have $\chi(\mathcal K_{i,j})\leq r_{i,j}$, and hence
		\[
		\sum_{x\in\mathcal K_{i,j}}(e_x-1)
		\geq
		\deg(\mathcal K_{i,j}\to D_i)-r_{i,j}.
		\]
		Summing over all components $j=1, \dots, m_i$, we get
		\[
		\sum_{x\in p^{-1}(D_i)}(e_x-1)
		\geq
		\sum_{j=1}^{m_i}\deg(\mathcal K_{i,j}\to D_i)
		-\sum_{j=1}^{m_i}r_{i,j}
		=
		d-r_i.
		\]
		Finally, applying the global Riemann--Hurwitz formula to $p: X \to Y$, we obtain
		\begin{align*}
			2g(X)-2
			&=
			d\bigl(2g(Y)-2\bigr)+\sum_{x\in X}(e_x-1)\\
			&\geq
			d\bigl(2g(Y)-2\bigr)
			+\sum_{i=1}^n\sum_{x\in p^{-1}(D_i)}(e_x-1)\\
			&\geq
			d\bigl(2g(Y)-2\bigr)+\sum_{i=1}^n(d-r_i)\\
			&=
			d\bigl(2g(Y)-2+n\bigr)-\sum_{i=1}^n r_i.
		\end{align*}
		This proves the desired inequality.
	\end{proof}
	
	\begin{remark}
		The lemma also holds if $X=\coprod_{j=1}^m X_j$ is a finite
		disjoint union of compact Riemann surfaces and $p$ is finite.
		In this case, we set $g(X)\colonequals\sum\limits_{j=1}^m g(X_j)$.
		In fact, we apply the lemma to each $X_j$ and sum the
		inequalities. Since
		\[\sum_{j=1}^m(2g(X_j)-2)=2g(X)-2m\leq2g(X)-2,\]
		we obtain the desired inequality.
	\end{remark}

	\section{Proof of the main theorem}\label{Sec:proof}
	In this section, let $k$ be a finitely generated field over
	$\mathbb{Q}$ of transcendence degree $e$, and let $B$ be a smooth,
	geometrically connected curve over $k$, with function field
	$K\colonequals k(B)$. In particular, $k$ is algebraically closed
	in $K$. Let $\pi:\mathcal{A}\to B$ be an abelian scheme of
	relative dimension $g$, and denote its generic fiber by $A$.
	
	Fix an ample symmetric line bundle $L\in\operatorname{Pic}(A)$.
	As explained in Section~\ref{Sec:ht}, a big and nef adelic line
	bundle $\overline{H}$ on $B$ determines a N\'eron--Tate height
	\[
	\widehat{h}_L^{\overline{H}}:
	A(\overline{K})\longrightarrow\mathbb{R}.
	\]
	Let $(x_n)_{n\geq 1}$ be a sequence in $A(\overline{K})$, with
	$d_n\colonequals\deg(x_n)=[K(x_n):K]$. Denote by
	\[
	C_n\colonequals\Zar{\{x_n\}}\subset\mathcal{A}
	\]
	the Zariski closure of $x_n$, which is a multisection of $\pi$
	of degree $d_n$. We write $g(C_n)$ for the sum of the genera of
	the geometric components of its normalized smooth compactification.
	
	Assume that $(x_n)_{n\geq 1}$ is generic and small. Namely, every infinite
	subsequence is Zariski dense in $A$, and
	$\widehat{h}_L^{\overline{H}}(x_n)\to 0$. By the Northcott property, $d_n$ tends to infinity as $n\to\infty$. Moreover, every subsequence remains generic and small.
	
	The goal of this section is to prove the following result.
	
	\begin{theorem}[=Theorem~\ref{thm:main1}]\label{thm:main}
		With the notation above, suppose that $(x_n)_{n\geq 1}$ is generic
		and small with respect to $\widehat{h}_L^{\overline{H}}$. If
		\[
		\liminf_{n\to\infty}\frac{g(C_n)}{d_n}=0,
		\]
		then the abelian scheme $\pi:\mathcal{A}\to B$ is isotrivial.
	\end{theorem}
	
	\subsubsection*{Settings}
	Let $V$ be a smooth projective model of $k$ over $\mathbb{Q}$.
	Thus $\dim V=e$ and $\mathbb{Q}(V)=k$.
	Let $\overline{B}$ be the smooth compactification of $B$, and
	choose a smooth projective model $\overline{\mathcal{B}}$
	equipped with a projective morphism
	$\pi_V:\overline{\mathcal{B}}\to V$
	whose generic fiber is $\overline{B}$.
	After spreading out, there exists an open subscheme
	$\mathcal{B}\subset\overline{\mathcal{B}}$ such that:
	\begin{enumerate}
		\renewcommand{\labelenumi}{(\roman{enumi})}
		\item
		the abelian scheme $\pi:\mathcal{A}\to B$ extends to an abelian
		scheme $\widetilde{\pi}:\ms{A}\to\mathcal{B}$;
		\item
		the generic fiber of $\mathcal{B}\to V$ is $B$.
	\end{enumerate}
	
	After replacing $L$ by a positive tensor power and shrinking
	the model if necessary, we may assume that $L$ extends to a
	symmetric relatively ample line bundle on $\ms{A}$.
	By Section~\ref{Sec:ht}, its canonical adelic extension
	$\overline{L}^{\mathrm{can}}\in
	\widehat{\operatorname{Pic}}(A)_{\mathrm{nef},\mathbb{Q}}$
	is defined over $\ms{A}\to\mathcal{B}$. By Proposition~\ref{prop:ht}(4), smallness is independent
	of the choice of $\overline H$. We therefore choose a big and nef adelic line bundle
	$\overline H\in \widehat{\mr{Pic}}(\overline{\mathcal B})$ such that $c_1(\overline H)_\infty$ is smooth and positive.
	
	After replacing $k$ by a finite extension, we may assume that
	\[
	\overline B\setminus B=\{s_1,\ldots,s_N\},
	\qquad s_i\in\overline B(k).
	\]
	The hypotheses and conclusion of
	Theorem~\ref{thm:main} are preserved under this base change.
	
	Choose a nonempty open subset $V^\circ\subset V$ such that
	$\pi_V$ is smooth over $V^\circ$ and each boundary point $s_i$
	extends to a section
	\[
	\widetilde{s}_i:V^\circ\longrightarrow\pi_V^{-1}(V^\circ).
	\]
	Thus the generic fiber of $\widetilde{s}_i$ is the point $s_i$.
	After shrinking $V^\circ$, we may assume that the images of
	these sections are pairwise disjoint.
	Restricting $\mathcal{B}$ accordingly, we may also assume that
	\[
	\mathcal{B}
	=
	\pi_V^{-1}(V^\circ)\setminus
	\bigcup_{i=1}^N\widetilde{s}_i(V^\circ).
	\]
	
	Choose a small polydisk $W\Subset V^\circ(\mathbb C)$
	and identify it with $\mathbb D(0,1)\subset\mathbb C^e$.
	For each $i$, choose a neighborhood $U_i$ of
	$\widetilde s_i(W)$ with coordinates
	\[
	U_i\simeq D(0,3)\times W,
	\]
	such that $\pi_V(z,v)=v$ and $\widetilde s_i(v)=(0,v)$.
	We choose the $U_i$ to be pairwise disjoint. For $b=(z,v)\in U_i\cap\mathcal B(\mathbb C)$,
	define the loop based at $b$ by
	\[
	\gamma_b:[0,2\pi]\longrightarrow
	(U_i)_v\setminus\{\widetilde s_i(v)\},
	\qquad
	t\longmapsto(ze^{\mathrm{i}t},v).
	\]
	
	Let $\mu_W$ be the standard Lebesgue measure on
	$W\simeq\mathbb{D}(0,1)$.
	Fix a closed polydisk $W_0\Subset W$ with $\mu_W(W_0)>0$,
	and set
	\[
	E_i\colonequals\{1\leq |z|\leq 2\}\times W_0
	\subset U_i\cap\mathcal{B}(\mathbb{C}).
	\]
	Each $E_i$ is compact. Equip it with the probability measure
	\[
	\mathrm{d}\mu_{E_i}(b)
	=
	\frac{\mathrm{d}r\,\mathrm{d}\theta\,\mathrm{d}\mu_W(v)}
	{2\pi\mu_W(W_0)},
	\qquad b=(re^{\mathrm{i}\theta},v).
	\]
	When $e=0$, we regard $W$ and $W_0$ as a single point of
	measure one.
	
	Viewing $A$ as the generic fiber of $\ms{A}$ over
	$\mathcal{B}$, let
	$\mathcal{C}_n\colonequals\Zar{\{x_n\}}\subset\ms{A}$.
	The generic fiber of $\mathcal{C}_n\to V^\circ$ is $C_n$,
	and the restriction of $\widetilde{\pi}$ induces a proper,
	generically finite morphism
	\[
	\pi_n:\mathcal{C}_n
	\longrightarrow\mathcal{B}
	\]
	of degree $d_n$.
	Choose a proper closed subscheme $Z_n\subset\mathcal B$
	such that $\pi_n$ is finite and \'etale of degree $d_n$
	over $\mathcal B\setminus Z_n$. For almost every $v\in W$, the sets
	$(Z_n)_v\colonequals Z_n(\mb{C})\cap\mathcal{B}(\mathbb{C})_v$
	are finite for all $n$. Indeed, for each $n$, the generic fiber of
	$Z_n\to V^\circ$ is finite, so $(Z_n)_v$ is finite
	outside a proper closed subset of $V^\circ$. The union of these exceptional subsets over all $n$
	has measure zero in $W$.
	
	Choose a symmetric relatively ample line bundle
	$\ms{L}\in\operatorname{Pic}(\ms{A})$ extending $L$ and
	rigidified along the zero section.
	Let $\omega\colonequals\omega(\ms{L})$ be its 
	Betti form, and define
	\[
	\omega_n
	\colonequals
	\widetilde{\pi}_*\bigl([\mathcal{C}_n]\wedge\omega\bigr).
	\]
	Then $\omega_n$ is a positive $(1,1)$-current on
	$\mathcal B(\mathbb C)$ and is smooth outside $Z_n$.
	Write $(\omega_n)_v$ for its restriction to
	$\mathcal B(\mathbb C)_v\setminus Z_n(\mb{C})_v$.

	\begin{lemma}\label{lem:slice}
		With the notation above,
		then for each $1\leq i\leq N$,
		\[
		\lim_{n\to\infty}
		\frac{1}{d_n}
		\int_{E_i}
		\omega_n\wedge\pi_V^*\mathrm{d}\mu_W
		=0.
		\]
		Here, we view $\mathrm{d}\mu_W$ as the standard $(e,e)$-form on $W$.
	\end{lemma}
	
	\begin{proof}
		Under the identification $W\simeq \mb{D}(0,1)$, viewing $\mathrm{d}\mu_W$ as the standard $(e,e)$-form on $W$, we have
		\[
		\mathrm{d}\mu_W
		=
		\left(\frac{\mathrm{i}}{2}\right)^e
		\mathrm{d}w_1\wedge\mathrm{d}\overline{w_1}
		\wedge\cdots\wedge
		\mathrm{d}w_e\wedge\mathrm{d}\overline{w_e}.
		\]
		
		Recall that $\overline{H}$ is a big and nef adelic line bundle on $\overline{\mc{B}}$ such
		that the $(1,1)$-form $\eta\colonequals c_1(\overline{H})_\infty$ is positive on $\ov{\mc{B}}(\mb{C})$. Since $E_i$ is compact in $\mathcal{B}(\mathbb{C})$,
		there exists $\varepsilon_1>0$ such that
		$\eta^e-\varepsilon_1\pi_V^*\mathrm{d}\mu_W$
		is semipositive on a neighborhood of $E_i$.
		By the projection formula,
		\[
		\begin{aligned}
			\frac{1}{d_n}
			\int_{E_i}\omega_n\wedge\pi_V^*\mathrm{d}\mu_W\leq
			\frac{1}{\varepsilon_1d_n}
			\int_{\mathcal{B}(\mathbb{C})}\omega_n\wedge\eta^e=
			\frac{1}{\varepsilon_1d_n}
			\int_{\ms{A}(\mathbb{C})}
			[\mathcal{C}_n]\wedge\omega
			\wedge\widetilde{\pi}^*\eta^e=
			\frac{1}{\varepsilon_1}
			\widehat{h}_{L,\mathrm{geom}}^H(x_n).
		\end{aligned}
		\]
		The last equality follows from the geometric height integral formula
		from Section~\ref{Sec:ht}. By Proposition~\ref{prop:htcomp}, there exists
		$\varepsilon_2>0$ such that
		\[
		\widehat{h}_{L,\mathrm{geom}}^H(x_n)
		\leq
		\frac{1}{\varepsilon_2}
		\widehat{h}_L^{\overline{H}}(x_n).
		\]
		for all $n$. Since $\lim\limits_{n\to\infty}\widehat{h}_L^{\overline{H}}(x_n)=0$, we obtain the conclusion.
	\end{proof}

	Fix $b\in E_i$ such that $\gamma_b\cap Z_n=\varnothing$.
	The loop $\gamma_b$ has $d_n$ lifts
	$\widetilde\gamma_{b,1},\ldots,\widetilde\gamma_{b,d_n}$
	to $\mathcal C_n(\mathbb C)$, one starting at each point
	of $\pi_n^{-1}(b)$.
	Let $r_n(b)$ be the number of connected components of
	$\pi_n^{-1}(\gamma_b)$. Each connected component is formed by joining the lifts end to end. Define
	\[
	\ell_n(b)
	\colonequals
	\frac{1}{d_n}
	\sum_{j=1}^{d_n}
	L_\omega\bigl(\widetilde{\gamma}_{b,j}\bigr).
	\]
	If $\gamma_b\cap Z_n\neq\varnothing$, we set
	$\ell_n(b)=+\infty$ and $r_n(b)=0$.
	
	\begin{remark}
		The function $r_n(b)$ and $\ell_n(b)$ is measurable. In fact, for almost every $v\in W$, the set $(Z_n)_v$ is finite. For such $v$, there are only finitely many $r\in [1,2]$ satisfies $\partial D(0,r)\times \{v\}$ meets $Z_n$. Thus, the set $b\in E_i$ with $\gamma_b\cap Z_n\neq \varnothing$ has $\mu_{E_i}$-measure zero, and the values on these points
		do not affect the integrals.
	\end{remark}

	\begin{lemma}\label{lem:average_length}
		With the notation above, for each $1\leq i\leq N$, we have
		\[
		\lim_{n\to\infty}
		\int_{E_i}\ell_n(b)\,\mathrm{d}\mu_{E_i}(b)=0.
		\]
	\end{lemma}
	
	\begin{proof}
		Fix $1\leq i\leq N$ and write
		$b=(re^{\mathrm i\theta},v)\in E_i$.
		If $\gamma_b\cap Z_n=\varnothing$,
		Lemma~\ref{lem:preimage} implies
		\[
		\ell_n(b)
		=
		\frac{1}{d_n}\sum_{j=1}^{d_n}
		L_\omega(\widetilde\gamma_{b,j})
		\leq
		\frac{1}{\sqrt{d_n}}
		L_{(\omega_n)_v}(\gamma_b).
		\]
		The loop $\gamma_b$ parametrizes
		$\partial D(0,r)\times\{v\}$, so its length is
		independent of $\theta$.
		
		For almost every $v\in W_0$, the set $(Z_n)_v$ is
		finite. Applying Lemma~\ref{lem:area}, we obtain
		\[
		\begin{aligned}
			\int_{E_i}\ell_n(b)^2\,\mathrm d\mu_{E_i}(b)
			&\leq
			\frac{1}{d_n\mu_W(W_0)}
			\int_{W_0}\int_1^2
			L_{(\omega_n)_v}
			\bigl(\partial D(0,r)\times\{v\}\bigr)^2
			\,\mathrm dr\,\mathrm d\mu_W(v)\\
			&\leq
			\frac{4\pi}{d_n\mu_W(W_0)}
			\int_{W_0}
			\left(
			\int_{(E_i)_v\setminus(Z_n)_v}(\omega_n)_v
			\right)
			\mathrm d\mu_W(v)\\
			&=
			\frac{4\pi}{d_n\mu_W(W_0)}
			\int_{E_i\setminus Z_n}
			\omega_n\wedge\pi_V^*\mathrm d\mu_W\\
			&\leq
			\frac{4\pi}{d_n\mu_W(W_0)}
			\int_{E_i}
			\omega_n\wedge\pi_V^*\mathrm d\mu_W.
		\end{aligned}
		\]
		Since $\mu_{E_i}$ is a probability measure, by
		Cauchy--Schwarz inequality, we obtain
		\[\left(\int_{E_i}\ell_n\,\mathrm{d}\mu_{E_i}\right)^2
		\leq\int_{E_i}\ell_n^2\,\mathrm{d}\mu_{E_i}\leq 
		\frac{4\pi}{d_n\mu_W(W_0)}
		\int_{E_i}
		\omega_n\wedge\pi_V^*\mathrm{d}\mu_W.\]
		By Lemma~\ref{lem:slice}, the last terms tends to $0$, which completes the proof.
	\end{proof}
	
	\begin{definition}
		For $1\leq i\leq N$, choose $b\in E_i$.
		The \emph{order} of $\pi:\mathcal{A}\to B$ at $s_i$,
		denoted $\ord(s_i)\in\mathbb{Z}_{>0}\cup\{\infty\}$,
		is the order of $\rho_b(\gamma_b)$, where
		\[
		\rho_b:\pi_1(\mathcal{B}(\mathbb{C}),b)
		\longrightarrow \mr{Aut}\big(H_1(\ms{A}(\mb{C})_b,\mb{Z})\big)\simeq\mathrm{GL}_{2g}(\mathbb{Z})
		\]
		is the monodromy representation associated with the Betti
		foliation of $\widetilde{\pi}:\ms{A}\to\mathcal{B}$.
	\end{definition}
	
	\begin{remark}
		For different choices of
		$b\in U_i\cap\mathcal{B}(\mathbb{C})$, the elements
		$\rho_b(\gamma_b)$ are conjugate in
		$\mathrm{GL}_{2g}(\mathbb{Z})$ Thus $\ord(s_i)$ is well-defined.
		
		By Proposition~\ref{prop:monotriv}, $\ord(s_i)=1$ if and only
		if $\mathcal{A}\to B$ has good reduction at $s_i$.
		Similarly, $\ord(s_i)<\infty$ if and only if it has potentially
		good reduction at $s_i$.
	\end{remark}
	
	The following proposition is a key ingredient in the proof
	of the main theorem. Using the equidistribution theorem, it gives an
	upper bound for $\mr{ord}(s_i)$ in terms of the averages of $r_n(b)/d_n$.
	
	\begin{proposition}\label{prop:fg_cycles}
		With the notation above, for each $1\leq i\leq N$, we have
		\begin{equation}\label{eq_fg_cycles}
			\limsup_{n\to\infty}
			\int_{E_i}\frac{r_n(b)}{d_n}\,
			\mathrm{d}\mu_{E_i}(b)
			\leq \frac{1}{\ord(s_i)},
		\end{equation}
		where $1/\infty=0$.    
	\end{proposition}
	
	\begin{proof}
		Fix $1\leq i\leq N$.
		If $\ord(s_i)=1$, the assertion follows immediately from
		$r_n(b)\leq d_n$ for almost all $b\in E_i$.
		We may therefore assume that $\ord(s_i)>1$.
		
		Fix an integer $q\geq 1$ such that $\rho_b(\gamma_b^j)\neq I_{2g}$ for all $b\in E_i$ and $1\leq j\le q$. For $b\in E_i$ with $\gamma_b\cap Z_n=\varnothing$, the fiber
		$\pi_n^{-1}(b)$ consists of $d_n$ points, and
		$\pi_n^{-1}(\gamma_b)$ has $r_n(b)$ connected components.
		Each component is a loop obtained by concatenating
		$m$ successive lifts of $\gamma_b$ for some $m\geq 1$,
		and meets $\pi_n^{-1}(b)$ in exactly $m$ points.
		
		Let $P_{n,q}(b)\subseteq\pi_n^{-1}(b)$ be the set of points
		contained in components for which $m\leq q$.
		If $\gamma_b\cap Z_n\neq\varnothing$, we set $P_{n,q}(b)=\varnothing$.
		
		Define a positive measure supported on
		$\widetilde{\pi}^{-1}(E_i)$ by
		\[
		\mu_{n,q}
		\colonequals
		\frac{1}{d_n}\int_{E_i}
		\left(\sum_{x\in P_{n,q}(b)}\delta_x\right)
		\mathrm{d}\mu_{E_i}(b).
		\]
		Its total mass satisfies
		\[
		\mu_{n,q}\bigl(\widetilde{\pi}^{-1}(E_i)\bigr)=\int_{E_i}\frac{\#P_{n,q}(b)}{d_n}\,
		\mathrm{d}\mu_{E_i}(b)
		\leq 1.
		\]
		
		For any $x\in\widetilde{\pi}^{-1}(E_i)$, set
		$b=\widetilde{\pi}(x)$ and define
		\[
		\operatorname{Dist}(x;q)
		\colonequals
		\min_{1\leq j\leq q}
		\operatorname{dist}_b
		\bigl(x,\rho_b(\gamma_b^j)(x)\bigr),
		\]
		where $\rho_b(\gamma_b^j)$ acts on the torus
		$\mathscr A(\mathbb C)_b$.
		For fixed $q$, this is a nonnegative continuous function on
		the compact space $\widetilde{\pi}^{-1}(E_i)$.
		
		\medskip
		\noindent\textbf{Claim 1.}
		We have
		\[
		\lim_{n\to\infty}
		\int_{\widetilde{\pi}^{-1}(E_i)}
		\operatorname{Dist}(x;q)\,
		\mathrm{d}\mu_{n,q}(x)=0.
		\]
		
		\noindent\textit{Proof of Claim 1.}
		Let $x\in P_{n,q}(b)$, and let
		$\Gamma\subseteq\pi_n^{-1}(\gamma_b)$ be the component containing
		$x$.
		Suppose that $\Gamma$ consists of $m$ successive lifts of
		$\gamma_b$, where $m\leq q$.
		By Lemma~\ref{lem:dist}, there is a constant
		$c_q>0$, independent of $n$, $b$, and $\Gamma$, such that
		\[
		\operatorname{Dist}(x;q)
		\leq
		\operatorname{dist}_b
		\bigl(x,\rho_b(\gamma_b^m)(x)\bigr)
		\leq c_qL_\omega(\Gamma).
		\]
		Each such $\Gamma$ contains at most $q$ points of $P_{n,q}(b)$. Consequently,
		\[
		\sum_{x\in P_{n,q}(b)}
		\operatorname{Dist}(x;q)
		\leq
		qc_q\sum_{\Gamma}L_\omega(\Gamma)
		=
		qc_qd_n\ell_n(b),
		\]
		where the sum runs over all connected components $\Gamma$ of $\pi_n^{-1}(\gamma_b)$.
		Hence, we obtain
		\begin{align*}
			\int_{\widetilde{\pi}^{-1}(E_i)}
			\operatorname{Dist}(x;q)\,
			\mathrm{d}\mu_{n,q}(x)
			&=
			\frac{1}{d_n}\int_{E_i}
			\sum_{x\in P_{n,q}(b)}
			\operatorname{Dist}(x;q)\,
			\mathrm{d}\mu_{E_i}(b)\\
			&\leq
			qc_q\int_{E_i}\ell_n(b)\,
			\mathrm{d}\mu_{E_i}(b).
		\end{align*}
		The last term tends to zero as $n\to\infty$ by
		Lemma~\ref{lem:average_length}, which completes the proof.\bqed
		
		\medskip
		\noindent\textbf{Claim 2.}
		The measures $\mu_{n,q}$ converge weakly to zero as
		$n\to\infty$.
		
		\noindent\textit{Proof of Claim 2.}
		View $\mathrm{d}\mu_{E_i}$ as a smooth $(e+1,e+1)$-form on $E_i$.
		Since $c_1(\overline H)_\infty$ is positive and $E_i$ is compact,
		there exists a constant $M>0$ such that $M c_1(\overline H)_\infty^{e+1}-\mathrm{d}\mu_{E_i}$ is semipositive on a neighborhood of $E_i$. Set
		\[
		\mu_n
		\colonequals
		\frac{1}{d_n}[\mathcal C_n]
		\wedge
		\widetilde{\pi}^{*}
		c_1(\overline H)_\infty^{e+1}.
		\]
		Viewing $\mu_{n,q}$ as a measure on $\mathscr A(\mathbb C)$
		by extension by zero, we have $\mu_{n,q}\leq M\mu_n$. The measures $\mu_{n,q}$ have mass at most one and are supported
		on the fixed compact set $\widetilde{\pi}^{-1}(E_i)$.
		From any subsequence of $(\mu_{n,q})_{n\geq 1}$, we may choose a further subsequence that converges weakly to a measure $\mu$. By Claim 1 and the continuity of
		$\operatorname{Dist}(\,\cdot\,;q)$, we obtain
		\[
		\int_{\widetilde{\pi}^{-1}(E_i)}
		\operatorname{Dist}(x;q)\,\mathrm{d}\mu(x)=0.
		\]
		Since $\operatorname{Dist}(x\,;q)$ is nonnegative, $\mu$ is supported on
		the closed set
		\[
		F_q
		\colonequals
		\left\{
		x\in\widetilde{\pi}^{-1}(E_i):
		\operatorname{Dist}(x;q)=0
		\right\}.
		\]
		To show $\mu=0$, it suffices to prove $\mu(F_q)=0$. In fact, for each $b\in E_i$, its fiber is
		\[
		(F_q)_b
		=
		\bigcup_{j=1}^{q}
		\ker\bigl(\rho_b(\gamma_b^j)-I_{2g}\bigr)
		\subseteq\mathscr{A}(\mathbb C)_b.
		\]
		By the choice of $q$, each matrix
		$\rho_b(\gamma_b^j)-I_{2g}$ is nonzero. Denote $\nu_b\colonequals
		\dfrac{\omega_b^g}{\deg_L(A)}$ the Haar measure on $A(\mathbb C)_b$. Then $\nu_b\bigl((F_q)_b\bigr)=0$ for almost all $b\in E_i$.
		
		By Theorem~\ref{thm:equidis}, the measures $\mu_n$ converge
		weakly to
		\[
		\mu_\infty
		\colonequals
		\frac{\omega^g}{\deg_L(A)}
		\wedge
		\widetilde{\pi}^{*}
		\bigl(c_1(\overline H)_\infty^{e+1}\bigr).
		\]
		Since $\mu_{n,q}\leq M\mu_n$ for all $n$, we also have $\mu\leq M\mu_\infty$. Since $\mu$ is supported on $F_q$, it follows that
		\[
		\mu\bigl(\widetilde{\pi}^{-1}(E_i)\bigr)
		=
		\mu(F_q)
		\leq M\mu_\infty(F_q)=M\int_{E_i}
		\nu_b\bigl((F_q)_b\bigr)\,
		c_1(\overline H)_\infty^{e+1}(b)=0.
		\]
		Hence $\mu=0$, and consequently $\mu_{n,q}$ converges weakly to zero as $n\to \infty$.\bqed
		
		\medskip
		We now estimate $r_n(b)$. For almost every $b\in E_i$, the number
		of components of $\pi_n^{-1}(\gamma_b)$ consisting of at most
		$q$ lifts is at most $\#P_{n,q}(b)$. Each of the other components
		contains at least $q+1$ points of $\pi_n^{-1}(b)$, so there are at most $\dfrac{d_n-\#P_{n,q}(b)}{q+1}$ such components. It follows that
		\[
		r_n(b)
		\leq
		\#P_{n,q}(b)
		+\frac{d_n-\#P_{n,q}(b)}{q+1}
		\leq
		\#P_{n,q}(b)+\frac{d_n}{q+1}.
		\]
		Taking integration, we obtain
		\[
		\int_{E_i}\frac{r_n(b)}{d_n}\,
		\mathrm{d}\mu_{E_i}(b)
		\leq \int_{E_i}\frac{\# P_{n,q}(b)}{d_n}\mr{d}\mu_{E_i}(b)+\frac{1}{q+1}=
		\mu_{n,q}\bigl(\widetilde{\pi}^{-1}(E_i)\bigr)
		+\frac{1}{q+1}.
		\]
		By Claim~2, $\lim\limits_{n\to\infty}\mu_{n,q}\bigl(\widetilde{\pi}^{-1}(E_i)\bigr)=0$. Thus
		\[
		\limsup_{n\to\infty}
		\int_{E_i}\frac{r_n(b)}{d_n}\,
		\mathrm{d}\mu_{E_i}(b)
		\leq \frac{1}{q+1}.
		\]
		
		If $2\leq\ord(s_i)<\infty$, taking $q=\ord(s_i)-1$ gives
		inequality. If $\ord(s_i)=\infty$, the inequality holds
		for every $q\geq1$. Letting $q\to\infty$ completes the proof.
	\end{proof}
	\begin{proposition}\label{prop:fg_kappa}
		With the notation above, we have
		\[2g(\overline B)-2+N
		-\sum_{i=1}^N\frac{1}{\ord(s_i)}
		\leq
		\liminf_{n\to\infty}\frac{2g(C_n)}{d_n},\]
		where $1/\infty=0$.
	\end{proposition}
	
	\begin{proof}
		Let $v\in W_0$ be a fully transcendental point.
		Then $v$ induces an embedding $k\hookrightarrow\mathbb C$.
		Through this embedding, we may view $(\mathcal C_n)_v$
		and $\mathcal B_v$ as complex curves.
		For each $n$, let $\overline{\mathcal C}_{n,v}$ denote
		the normalized smooth compactification of $(\mathcal C_n)_v$.
		The curve $\overline{\mathcal C}_{n,v}$ may be disconnected,
		and the sum of the genera of its connected components is
		$g(C_n)$. We also have a finite morphism
		\[
		\pi_{n,v}:\overline{\mathcal C}_{n,v}
		\longrightarrow\overline{\mathcal B}_v
		\]
		of degree $d_n$.
		
		Fix an integer $n\geq1$. For each $1\leq i\leq N$, choose $r_i\in[1,2]$ such that
		the circle $\partial D(0,r_i)\times\{v\}$ does not meet $Z_n$.
		Only finitely many $r_i$ are excluded.
		Let $\theta_i\in [0,2\pi)$, $b_i=(r_i e^{\mathrm{i}\theta_i},v)$, and $D_i=\overline{D(0,r_i)}\times\{v\}$. Then the disks $D_i$ are pairwise disjoint. Moreover, since $\gamma_{b_i}=\partial D_i$, $\pi_{n,v}^{-1}(\partial D_i)$ has $r_n(b_i)$ connected components.
		
		Applying Lemma~\ref{lem:Riehur} to each connected component
		$X$ of $\overline{\mathcal C}_{n,v}$ and summing, we obtain
		\[
		\begin{aligned}
			\sum_{i=1}^N r_n(b_i)
			&\geq
			d_n\bigl(2g(\overline B)-2+N\bigr)
			-\sum_X\bigl(2g(X)-2\bigr)\\
			&\geq
			d_n\bigl(2g(\overline B)-2+N\bigr)
			-\bigl(2g(C_n)-2\bigr).
		\end{aligned}
		\]
		Here we use $g(\overline{\mathcal B}_v)=g(\overline B)$
		and the fact that $\overline{\mathcal C}_{n,v}$ has at least
		one connected component.
		When $N=0$, the same inequality follows directly from
		Riemann--Hurwitz formula.
		
		Almost every $v\in W_0$ is fully transcendental, and
		$\gamma_b\cap Z_n=\varnothing$ for almost every $b\in E_i$.
		Dividing by $d_n$ and integrating with respect to $v,r_i$ and $\theta_i$, we obtain
		\[
		\begin{aligned}
			\sum_{i=1}^N
			\int_{E_i}\frac{r_n(b)}{d_n}\,
			\mathrm{d}\mu_{E_i}(b)
			&=
			\frac{1}{2\pi\mu_W(W_0)}
			\sum_{i=1}^N
			\int_{W_0}\int_1^2\int_0^{2\pi}
			\frac{r_n(b_i)}{d_n}\,
			\mathrm{d}\theta_i\,\mathrm{d}r_i\,\mathrm{d}\mu_W(v)\\
			&\geq
			2g(\overline B)-2+N
			-\frac{2g(C_n)-2}{d_n}.
		\end{aligned}
		\]
		Applying Proposition~\ref{prop:fg_cycles}, we have
		\begin{align*}
			\liminf_{n\to\infty}\frac{2g(C_n)}{d_n}
			&\geq
			2g(\overline B)-2+N
			-\sum_{i=1}^N
			\limsup_{n\to\infty}
			\int_{E_i}\frac{r_n(b)}{d_n}\,
			\mathrm{d}\mu_{E_i}(b)\\
			&\geq
			2g(\overline B)-2+N
			-\sum_{i=1}^N\frac{1}{\ord(s_i)}.
		\end{align*}
		This completes the proof.
	\end{proof}
	
	\begin{proof}[Proof of Theorem~\ref{thm:main}]
		Let $v\in W_0$ be a fully transcendental point, inducing an
		embedding $k\hookrightarrow\mathbb C$. Via this embedding, we
		identify $\mathscr A_v\to\mathcal B_v$ with
		$\mathcal A_{\mathbb C}\to B_{\mathbb C}$. It suffices to show
		that this abelian scheme is isotrivial. For each $i$, let
		$s_{i,\mathbb C}\in\overline B_{\mathbb C}$ be the point
		corresponding to $\widetilde s_i(v)$ under this identification.

		Without loss of generality, we may assume that
		$\ord(s_i)<\infty$ for $1\leq i\leq N_1$ and
		$\ord(s_i)=\infty$ for $N_1<i\leq N$, for some
		$0\leq N_1\leq N$.
		
		For each $i$, choose a base point $b_i\in(E_i)_v$.
		The local monodromy $\rho_{b_i}(\gamma_{b_i})$ has order
		$\ord(s_i)$. Fix a base point $b_0\in B_{\mathbb C}$ and
		paths from $b_0$ to each $b_i$. These paths identify each $\gamma_{b_i}\in \pi_1(B_{\mathbb C},b_i)$ with a loop
		$\gamma_i\in\pi_1(B_{\mathbb C},b_0)$.
		Write $\rho=\rho_{b_0}$ for the monodromy representation.
		Then $\rho(\gamma_i)$ also has order $\ord(s_i)$.
		
		Choose a prime $p$ sufficiently large such that
		\[
		\rho(\gamma_i^j)\not\equiv I_{2g}\pmod p
		\]
		for all $1\leq i\leq N_1$ and $1\leq j<\ord(s_i)$. Let $G$ be the kernel of the
		homomorphism
		\[
		\pi_1(B_{\mathbb C},b_0)
		\xrightarrow{\ \rho\ }
		\mathrm{GL}_{2g}(\mathbb Z)
		\longrightarrow
		\mathrm{GL}_{2g}(\mathbb F_p).
		\]
		Then $G$ is a normal subgroup of finite index in
		$\pi_1(B_{\mathbb C},b_0)$. By the Riemann existence theorem,
		it corresponds to a connected finite \'etale Galois cover $\varphi:B'\longrightarrow B_{\mathbb C}$.
		
		Let $\overline{B'}$ be the smooth compactification of $B'$,
		and let $e_i$ denote the order of the image of
		$\rho(\gamma_i)$ in $\mathrm{GL}_{2g}(\mathbb F_p)$.
		The induced finite morphism $\overline\varphi:\overline{B'}\longrightarrow\overline B_{\mathbb C}$
		has exactly $\deg(\ov{\varphi})/e_i$ points above $s_{i,\mathbb C}$,
		each with ramification index $e_i$. By the choice of $p$, $e_i=\ord(s_i)$ for $1\leq i\leq N_1$.
		
		Let $S\subset\overline{B'}$ be the set of points lying above
		$s_{i,\mathbb C}$ for $N_1<i\leq N$, and set $U=\overline{B'}\setminus S$. Then
		\[
		\#S=\deg(\varphi)\sum_{i=N_1+1}^N\frac{1}{e_i}.
		\]
		Consider the complex abelian scheme
		\[
		\mathcal A':=\mathcal A_{\mathbb C}\times_{B_{\mathbb C}}B'
		\longrightarrow B'.
		\]
		At each point $s\in U\setminus B'$ lying above
		$s_{i,\mathbb C}$ with $1\leq i\leq N_1$, the local monodromy
		of $\mathcal A'$ is conjugate to $\rho(\gamma_i)^{e_i}$
		and is therefore trivial, since $e_i=\ord(s_i)$.
		By Proposition~\ref{prop:monotriv}, $\mathcal A'\to B'$
		extends to an abelian scheme $\mathcal A''\to U$.
		
		By the Riemann--Hurwitz formula,
		\[
		2g(\overline{B'})-2
		=
		\deg(\ov{\varphi})
		\left(
		2g(\overline B)-2
		+\sum_{i=1}^N\left(1-\frac{1}{e_i}\right)
		\right).
		\]
		Consequently,
		\[
		\begin{aligned}
			2g(\overline{B'})-2+\#S
			&=
			\deg(\ov{\varphi})
			\left(
			2g(\overline B)-2+N
			-\sum_{i=1}^{N_1}\frac{1}{e_i}
			\right)\\
			&=
			\deg(\ov{\varphi})
			\left(
			2g(\overline B)-2+N
			-\sum_{i=1}^N\frac{1}{\ord(s_i)}
			\right)\\
			&\leq  
			\deg(\ov{\varphi})\liminf_{n\to\infty}\frac{2g(C_n)}{d_n}=0
		\end{aligned}
		\]
		Thus either $g(\overline{B'})=0$ and $\#S\leq 2$, or
		$g(\overline{B'})=1$ and $S=\varnothing$. Hence $U$ is
		isomorphic to $\mathbb P^1_{\mathbb C}$,
		$\mathbb A^1_{\mathbb C}$, $\mathbb G_{m,\mathbb C}$,
		or an elliptic curve. By Proposition~\ref{prop:isotriv}, $\mathcal A''\to U$ is
		isotrivial. Since $B'\to B_{\mathbb C}$ is finite \'etale,
		the abelian scheme $\mathcal A_{\mathbb C}\to B_{\mathbb C}$, and hence
		$\mathcal A\to B$, is isotrivial.
	\end{proof}

	\section{Applications to the uniform boundedness conjecture}\label{Sec_app}
	In this subsection, let $k$ be a finitely generated
	field over $\mathbb Q$, and let $B$ be a smooth, geometrically connected curve over $k$.
	Set $K\colonequals k(B)$. In particular, $k$ is algebraically closed in $K$. 
	Let $A$ be an abelian variety over $K$ of dimension $g$.
	
	Given an ample and symmetric line bundle $L \in \operatorname{Pic}(A)$ and a big and nef adelic line bundle $\overline{H}$ on $B$, we may define the Néron--Tate height
	\[
	\widehat{h}_L^{\overline{H}} : A(\overline{K}) \longrightarrow \mathbb{R}
	\]
	as defined in Section~\ref{Sec:ht}.
	
	For any point $x \in A(\overline{K})$, let $C_x$ be the smooth projective curve over $k$ with function field $K(x)$. We establish the following Bogomolov-type result:
	
	\begin{theorem}\label{thm:bogotype}
		With the notations as above, assume $A$ contains no non-trivial isotrivial abelian subvariety. Then, there exists constants $\varepsilon,\delta>0$ such that the set
		\[
		\left\{ x \in A(\overline{K}) : \widehat{h}_L^{\overline{H}}(x) \leq \varepsilon,\quad \frac{g(C_x)}{[K(x):K]} \leq \delta \right\}
		\]
		is finite.
	\end{theorem}
	
	\begin{proof}
		Suppose there exists a sequence of distinct
		points $(x_n)_{n\geq1}$ in $A(\overline K)$ such that $\dfrac{g(C_{x_n})}{[K(x_n):K]}\to 0$ and $\widehat h_L^{\overline H}(x_n)\to 0$. After passing to a subsequence, we may assume that
		$(x_n)_{n\geq 1}$ is generic in an irreducible closed subvariety
		$V\subset A_{\overline K}$ of positive dimension.
		
		By the Bogomolov conjecture over finitely generated fields, proved by Moriwaki \cite[Theorem 8.1]{Mor00}, there exist an abelian subvariety
		$A_1\subset A_{\overline K}$ and a torsion point
		$a\in A(\overline K)$ such that
		\[
		V=A_1+a.
		\]
		Choose an integer $\ell\geq1$ with $\ell a=0$, and set $y_n=\ell x_n$.
		Then $y_n\in A_1$, and $(y_n)_{n\geq 1}$ is generic in $A_1$.
		Since $[\ell]$ is finite of degree $\ell^{2g}$, $[K(x_n):K(y_n)]\leq \ell^{2g}$. Moreover, we may pass to a subsequence such that $y_n$ are distinct.
		
		Let $C_{y_n}$ be the smooth projective curve over $k$ with function field $K(y_n)$. The inclusion $K(y_n)\subset K(x_n)$ induces a finite
		surjective morphism $C_{x_n}\to C_{y_n}$. 
		It follows that
		\[
		\frac{g(C_{y_n})}{[K(y_n):K]}\leq \frac{g(C_{x_n})}{[K(y_n):K]}\leq\ell^{2g}\frac{g(C_{x_n})}{[K(x_n):K]}.
		\]
		Moreover, $\widehat h_L^{\overline H}(y_n)
		=\ell^2\widehat h_L^{\overline H}(x_n)\to 0$.
		
		Let $K_1/K$ be the smallest field over which
		$A_1\subset A_{\overline K}$ is defined.
		We claim that $K_1\subset K(y_n)$ for all sufficiently large $n$.
		In fact, for any $\sigma\in \Gal(\ov{K}/K)\setminus \Gal(\ov{K}/K_1)$, $\sigma(A_1)\cap A_1$ is a proper closed subset of $A_1$. Since $(y_n)$ is generic in $A_1$, $y_n\notin \sigma(A_1)$ for all $n$ sufficient large and all $\sigma\in \Gal(\ov{K}/K)\setminus \Gal(\ov{K}/K_1)$. On the other hand, for $\sigma\in \Gal(\ov{K}/K(y_n))$, $y_n\in A_1\cap \sigma(A_1)$. This implies $\sigma(A_1)=A_1$. So $\sigma$ fixes $K_1$, which implies $K_1\subset K(y_n)$.
		
		Let $k_1$ be the algebraic closure of $k$ in $K_1$. Choose a smooth curve $B_1$ over $k_1$ with function field $K_1$ and a morphism $B_1\to B_{k_1}$. Shrinking $B$ and $B_1$ if necessary, $A$ extends to an abelian scheme $\mc{A}\to B$ and $A_1$ extends to an abelian scheme $\mc{A}_1\subset \mc{A}\times_BB_1$. Let $C_n'$ be the Zariski closure of $y_n$ in $\mathcal A_1$. Since $K_1\subset K(y_n)$, we have $K_1(y_n)=K(y_n)$. Note that the normalized compactification of connected components of $(C_n')_{\ov{k_1}}$ is a subset of connected components of $(C_{y_n})_{\ov{k}}$. Then
		\[
		\frac{g(C_n')}{[K_1(y_n):K_1]}\leq [K_1:K]\frac{g(C_{y_n})}{[K(y_n):K]}\leq [K_1:K]\ell^{2g}\frac{g(C_{x_n})}{[K(x_n):K]}.
		\]
		Consequently, $\lim\limits_{n\to\infty}\dfrac{g(C_n')}{\deg(C_n'\to B_1)}=0$. Theorem~\ref{thm:main1} therefore implies that $\mc{A}_1$ is isotrivial, which implies $A_1$ is isotrivial. Since $\dim A_1=\dim V>0$, this contradicts the assumption. 
	\end{proof}

	For each $x\in A(\overline K)$, let $C_x^\circ$ be the
	smooth projective curve over $\overline k$ with function field $\overline kK(x)$. For a positive integer $n$, let $A[n]^{\times}$ denote the set of torsion points in $A(\overline{K})$ of exact order $n$. 
	We obtain the following corollary, which answers a conjecture proposed in \cite{CT11}.
	
	\begin{corollary}\label{cor:infgenus}
		Assume $A$ contains no non-trivial isotrivial abelian subvariety. Let $g(n) \colonequals  \min\limits_{x \in A[n]^{\times}} g(C_x^{\circ})$. Then 
		\[ \lim_{n \to \infty} g(n) = +\infty. \]
	\end{corollary}
	
	\begin{proof}
		Assume that there exists an integer $N > 0$ and a sequence of distinct torsion points $x_n \in A(\ov{K})$, such that $g(C_{x_n}^{\circ}) \leq N$ for all $n \geq 1$. Denote $K_1:=\ov{k}(B)$. By \cite[Proposition B.1]{GH19}, the order of any point in
		\[\{x\in A_{K_1}(\ov{K_1})_{\mr{tors}}: [K_1(x):K_1]\leq D\}\]
		is bounded. It follows that $[K_1(x_n):K_1]\to\infty$ as $n\to\infty$. Since the connected componets of $(C_x)_{\ov{k}}$ are conjugate, they have the same genus and the the same degree over $\ov{B}_{\ov{k}}$. Hence
		\[
		\frac{g(C_{x_n})}{[K(x_n):K]}
		=
		\frac{g(C_{x_n}^\circ)}{[K_1(x_n):K_1]}
		\leq
		\frac{N}{[K_1(x_n):K_1]}
		\longrightarrow0.
		\]
		Since $x_n$ are torsion points, their Néron--Tate heights satisfy $\widehat{h}_L^{\overline{H}}(x_n) = 0$ for all $n \geq 1$. This contradicts Theorem~\ref{thm:bogotype}.
	\end{proof}
	
	Let $\mc{A}\to B$ be an abelian scheme with generic fiber $\mc{A}_K=A$.The following theorem can be regarded as a generalization of Theorem~\ref{thm:CT}.
	
	\begin{theorem}[=Theorem~\ref{thm:sec1}]\label{thm:sec}
		With notations as above, assume $A$ contains no non-trivial isotrivial abelian subvariety. Let $P: B \to \mathcal{A}$ be a section. Then for any finite extension $k'/k$ and any prime $\ell$, there exists an integer $N \colonequals  N(\mathcal{A}, P,k, k', \ell)$ such that for any $b \in B(k')$, 
		\[ \#\left\{z \in \mathcal{A}_b(k') : \ell^n z = P_b \text{ for some } n \geq 0 \right\} \leq N. \]
	\end{theorem}
	
	\begin{proof}
		After replacing $\mc{A}$ and $B$ by $\mc{A}\otimes_k k^{\prime}$ and $B\otimes_k k^{\prime}$, we may assume $k^{\prime}=k$. Let $x_0 \in A(K)$ be the point in the generic fiber corresponding to the section $P$. Consider the set 
		\[ T = \{ x \in A(\overline{K}) : \ell^n x = x_0 \text{ for some } n \geq 0 \text{ and } g(C_x) \leq 1 \}\cup \{x_0\}. \]
		If $T$ is infinite, we may choose a sequence of distinct points $(x_n)_{n\geq0}$ in $T$ such that $\ell x_{n+1}=x_n$ for all $n\geq0$. This implies 
		\[ \lim_{n \to \infty} \widehat{h}_L^{\overline{H}}(x_n) = \lim_{n \to \infty} \frac{1}{\ell^{2n}} \widehat{h}_L^{\overline{H}}(x_0) = 0. \]
		By the Northcott property, $[K(x_n):K]\to \infty$. It follows that
		\[\frac{g(C_{x_n})}{[K(x_n):K]}\leq \frac{1}{[K(x_n):K]}\to 0.\]
		Theorem~\ref{thm:bogotype} then implies that the $\mr{Tr}_{\ov{K}/\ov{k}}(A_{\ov{K}})$ is non-trivial, which contradicts our assumption. Thus, $T$ is finite.
		
		Now define 
		\[S \colonequals  \{ x \in A(\overline{K}) : \ell x \in T \text{ and } x \notin T \}.\] Since $T$ is finite, $S$ is also finite. By the definition of $T$, every $x \in S$ must satisfy $g(C_x) \geq 2$. For any point $z \in \mathcal{A}(\overline{k})$, if $\ell^n z \in P$ for some $n \geq 0$, then either $z\in C_x$ for some $x \in T$, or $\ell^m z \in C_x$ for some $x \in S$ and $m \geq 0$.
		
		Let $C_x':=\Zar{\{x\}}\subset \mc{A}$ be the multi-section. Each $C_x'$ is a connected component of $[\ell^n]^{-1}(P)$ for some $n$, so $C_x'\to B$ is finite \'etale and $C_x'$ is smooth. So $C_x'$ is an open subset of $C_x$ for each $x$. Every $x\in S$ satisfies $g(C_x')=g(C_x)\geq2$. If $C_x'$ is geometrically connected, then by Faltings' Theorem \cite[Chap. VI, Theorem 3]{FW92}, the set of rational points $C_x'(k)$ is finite. Otherwise, $C_x'(k)=\varnothing$. Thus, the set
		\[ V \colonequals  \bigcup_{x \in S} \operatorname{Im}\big( C_x'(k) \to B(k) \big) \subset B(k), \]
		is a finite. 
		
		For $b\in B(k)$, write
		\[
		D_b=
		\left\{
		z\in\mathcal A_b(k):
		\ell^n z=P_b\text{ for some }n\geq0
		\right\}.
		\]
		We claim that, if $b\notin V$, then $D_b\subseteq\bigcup\limits_{x\in T}(C_x')_b(k)$. In fact, let $z\in D_b$ with $\ell^n z=P_b$. Choose $x\in[\ell^n]^{-1}(x_0)$ such that $z\in C_x'$.
		If $x\notin T$, then $y=\ell^jx\in S$ for some $j\geq 0$. It follows that $\ell^jz\in C_y'(k)$. This implies $b\in V$, which is a contradiction. Therefore, 
		\[\#D_b\leq \sum_{x\in T}\#(C_x')_b(k)\leq \sum_{x\in T}[K(x):K]\]
		for $b\notin V$.
		
		For each $b\in V$, we claim that $D_b$ is finite. In fact, by Mordell-Weil theorem, the group $\mathcal A_b(k)$ is finitely generated. If $P_b$ is non-torsion, then the image of $P_b$ in the free abelian group $A_b(k)/A_b(k)_{\mr{tors}}$ is divisible by only
		finitely many powers of $\ell$. So each equation
		$[\ell^n]z=P_b$ has finitely many solutions. If $P_b\in A_b(k)_{\mr{tors}}$, then $D_b\subset A_b(k)_{\mr{tors}}$. This proves the claim. 
		
		Set
		\[
		N=
		\max_{b\in V}\#D_b
		+\sum_{x\in T}[K(x):K].
		\]
		For every $b\in B(k)$, all points $z\in\mathcal A_b(k')$ satisfying
		$\ell^n z=P_b$ belong to $D_b$. The preceding bounds therefore prove the theorem.
	\end{proof}
	
	When the section $P$ is non-torsion, Theorem~\ref{thm:sec} generally fails for abelian schemes that contain an isotrivial abelian subscheme, as shown in Example~\ref{exam:nontors}. If the section $P$ is torsion, it suffices to consider the case where $P$ is the zero section. The following lemma provides the necessary bound for the isotrivial case.
	
	\begin{lemma}\label{lem:isotriv}
		Assume the abelian scheme $\mathcal{A} \to B$ is isotrivial. For any integer $d > 0$, there exists an integer $N \colonequals  N(\mathcal{A}, B, d) > 0$ such that 
		\[ \#\mathcal{A}_b(k')_{\mr{tors}} \leq N \]
		for any finite extension $k'/k$ with $[k':k] \leq d$ and any $b \in B(k')$.
	\end{lemma}
	
	\begin{proof}
		Since $\mathcal{A} \to B$ is isotrivial, there exists a finite cover $p: B' \to B$, a finite extension $k_0/k$, and an abelian variety $A_0$ over $k_0$ such that $(\mathcal{A} \times_B B') \otimes_k k_0 \simeq A_0 \times_{\operatorname{Spec} k_0} B'_{k_0}$. Let $r = \deg(p)$. By the Northcott property, the set
		\[
		T=\left\{
		x\in A_0(\overline k)_{\mathrm{tors}}:
		[k_0(x):k_0]\leq rd
		\right\}
		\]
		is finite
		
		Let $[k':k]\leq d$ and $b\in B(k')$. Set $F_0=k'k_0$ and choose a point $b'\in B'(F)$ above $b$, where $F/F_0$ is a finite extension of degree at most $r$. Then
		\[
		[F:k_0]= [F:F_0][F_0:k_0]\leq rd.
		\]
		Taking the fiber at $b'$ gives $\mathcal A_b\otimes_{k'}F\simeq A_0\otimes_{k_0}F$. Then we have
		\[
		\#\mathcal A_b(k')_{\mathrm{tors}}
		\leq \#A_0(F)_{\mathrm{tors}}
		\leq \#T.
		\]
		Taking $N:=\# T$ completes the proof.
	\end{proof}
	
	Here we give a new proof of Theorem~\ref{thm:CT}.
	
	\begin{theorem}\label{thm:CT1}
		Let $k$ be a finitely generated field over $\mb{Q}$ and $B$ be a smooth, geometrically connected curve over $k$. Let $\pi: \mc{A} \to B$ be an abelian scheme. For any finite extension $k'/k$ and any prime $\ell$, there exists an integer $N \colonequals  N(\mc{A}, B, k, k', \ell)$ such that for all $b \in B(k')$, 
		\[ \#\mc{A}_b(k')[\ell^{\infty}] \leq N, \]
		where $\mc{A}_b(k')[\ell^{\infty}] \colonequals  \bigcup\limits_{n=0}^{\infty} \mc{A}_b(k')[\ell^n]$ is the set of $\ell$-primary torsion points in $\mc{A}_b(k')$.
	\end{theorem}
	
	\begin{proof}
		We proceed by induction on the relative dimension $g = \dim(\mc{A}/B)$. The case $g=0$ is trivial. Assume the theorem holds for all abelian schemes of relative dimension less than $g$.
		
		\noindent\textbf{Case 1.} If $\mc{A}$ contains no non-trivial isotrivial abelian subscheme. Let $P: B \to \mc{A}$ be the zero section. We have
		\[ \mc{A}_b(k')[\ell^{\infty}] = \{ z \in \mc{A}_b(k') : \ell^n z = P_b \text{ for some } n \geq 0 \}. \]
		The conclusion follows directly from Theorem~\ref{thm:sec}.
		
		\noindent\textbf{Case 2.} If $\mc{A}$ contains an isotrivial abelian subscheme $\mc{A}_1$, let $\overline{\mc{A}} \colonequals  \mc{A}/\mc{A}_1$ be the quotient abelian scheme. Then $0\leq \dim(\overline{\mc{A}} /B) < g$. By the induction hypothesis, there exists a constant $N_1 = N_1(\overline{\mc{A}}, B, k', \ell)$ such that $\#\overline{\mc{A}}_b(k')[\ell^{\infty}] \leq N_1$ for all $b \in B(k')$. By Lemma~\ref{lem:isotriv}, there exists $N_2 = N_2(\mc{A}_1, B, [k':k])$ such that $\#\mc{A}_{1,b}(k')_{\mr{tors}} \leq N_2$ for all $b \in B(k')$.
		
		Consider the exact sequence of abelian groups of $\ell$-primary torsion points:
		\[ 0 \to \mc{A}_{1,b}(k')[\ell^{\infty}] \to \mc{A}_b(k')[\ell^{\infty}] \to \overline{\mc{A}}_b(k')[\ell^{\infty}]. \]
		It follows that 
		\[ \#\mc{A}_b(k')[\ell^{\infty}]\leq \#\mc{A}_{1,b}(k')[\ell^{\infty}] \cdot \#\overline{\mc{A}}_b(k')[\ell^{\infty}] \leq N_2 \cdot N_1. \]
		This completes the proof.
	\end{proof}

	\appendix
	
	\section{Further results}\label{sec:appendix}
	The main goal of this appendix is to prove the following
	theorem, which strengthens Theorem~\ref{thm:sec1}.
	
	\begin{theorem}\label{thm:sec_strong}
		Let $k$ be a finitely generated field over $\mb{Q}$, $B$ be a smooth geometrically connected curve over $k$, and $\pi:\mc{A}\to B$ be an abelian scheme. Assume $\mc{A}$ contains no nontrivial isotrivial abelian subscheme. Let $P:B\to\mathcal A$ be a section.
		Then, for every integer $d\geq1$ and every prime $\ell$,
		there exists an integer
		$N=N(\mathcal A/B,P,k,d,\ell)$ such that, for every
		finite extension $k'/k$ with $[k':k]\leq d$ and every
		$b\in B(k')$,
		\[
		\#\left\{
		z\in\mathcal A_b(k'):
		\ell^n z=P_b\text{ for some }n\geq0
		\right\}\leq N.
		\]
	\end{theorem}
	
	We deduce this theorem from Theorem~\ref{thm:bogotype}
	using gonality and Weil restriction. The argument is similar to that in \cite{CT13}.
	
	\subsection{Gonality of a curve}
	Let $\mathbf{k}$ be an algebraically closed field of characteristic zero. For an irreducible quasi-projective curve over $\mathbf{k}$, we define the gonality $\mr{gon}(C)$ to be the minimal degree of a dominant rational map from $C$ to $\mb{P}^1_{\mathbf{k}}$. In this section, we prove some properties of the gonality.
	
	For a smooth projective curve $C$ over $\mathbf{k}$, define
	\[
	W_r^1(C)\colonequals
	\{L\in\operatorname{Pic}^r(C):h^0(C,L)\geq2\}.
	\]
	By semicontinuity, $W_r^1(C)$ is a closed subset
	of the Picard variety $\operatorname{Pic}^r(C)$.
	
	\begin{lemma}\label{lem:gon_pencil}
		Let $C$ be a smooth projective curve over $\mathbf{k}$ with
		$\mr{gon}(C)=r$. Then every $L\in W_r^1(C)$ is base-point-free and satisfies
		$h^0(C,L)=2$.
	\end{lemma}
	
	\begin{proof}
		Let $L\in W_r^1(C)$, and suppose that $P\in C$ is a base
		point of $L$. Then
		\[
		h^0(C,L(-P))=h^0(C,L)\geq2.
		\]
		Choose linearly independent sections
		$s_0,s_1\in H^0(C,L(-P))$, and consider the divisor
		\[
		D\colonequals
		\sum_{Q\in C}
		\min\{\ord_Q(s_0),\ord_Q(s_1)\}[Q]\in \mr{Div}(C).
		\]
		View $s_0,s_1$ as sections of $L(-P-D)$, they have no common zero. Then we obtain a nonconstant morphism
		\[
		f:C\to\mb{P}^1,\qquad
		x\mapsto[s_0(x):s_1(x)],
		\]
		with $f^*\mathcal O_{\mb{P}^1}(1)\simeq L(-P-D)$. Consequently,
		\[
		\deg(f)=\deg(L)-1-\deg(D)\leq r-1.
		\]
		Contradiction. Thus $L$ is base-point-free.
		
		If $h^0(C,L)\geq3$, then for any $P\in C$,
		\[
		h^0(C,L(-P))\geq h^0(C,L)-1\geq2.
		\]
		The contradiction follows from the same argument. Hence $h^0(C,L)=2$.
	\end{proof}
	
	\begin{lemma}\label{lem:gon_norm}
		Let $f:C\to B$ be a surjective morphism of smooth projective
		curves over $\mathbf{k}$ with $\mr{gon}(C)=\mr{gon}(B)=r$.
		Then the norm map restricts to a morphism
		\[
		\operatorname{Nm}_f:W_r^1(C)\longrightarrow W_r^1(B).
		\]
	\end{lemma}
	
	\begin{proof}
		Let $L\in W_r^1(C)$ and set $M\colonequals\operatorname{Nm}_f(L)$.
		By Lemma~\ref{lem:gon_pencil}, a basis of $H^0(C,L)$
		defines a morphism $\varphi:C\to\mathbb P^1$ of degree $r$. 
		
		We claim that $M$ is base-point-free.
		For any $b\in B$, choose $x\in\mathbb P^1\setminus\varphi(f^{-1}(b))$.
		The effective divisor $D_x\colonequals f_*\varphi^*(x)$ is
		of degree $r$ and satisfies $\mathcal O_B(D_x)\simeq M$.
		Choose a section $s\in H^0(B,M)$ with $\mr{div}(s)=D_x$.
		Since $b\notin\operatorname{Supp}(D_x)$, we have $s(b)\neq0$. This proves the claim. 
		
		Since $M$ is base-point-free and $\deg M=r>0$, we have $h^0(B,M)\geq2$, so $M\in W_r^1(B)$.
		Hence the norm map induces a morphism $W_r^1(C)\to W_r^1(B)$.
	\end{proof}

	\begin{lemma}\label{lem:gon_cartesian}
		Let $f:C\to B$ be a surjective morphism of smooth projective
		curves over $\mathbf{k}$ with $\mr{gon}(C)=\mr{gon}(B)=r$.
		Then, for any morphism $\varphi_C:C\to\mathbb P^1$ of degree $r$,
		there exist morphisms $\varphi_B:B\to \mb{P}^1$ of degree $r$ and $g:\mb{P}^1\to \mb{P}^1$ such that the following
		diagram commutes:
		\[
		\begin{tikzcd}
			C \arrow[r, "\varphi_C"] \arrow[d, "f"']
			& \mathbb P^1 \arrow[d, "g"] \\
			B \arrow[r, "\varphi_B"']
			& \mathbb P^1.
		\end{tikzcd}
		\]
		Moreover, the fiber product $B\times_{\mathbb P^1,g}\mathbb P^1$ is integral,
		and the induced morphism $(f,\varphi_C):C\to B\times_{\mathbb P^1,g}\mathbb P^1$
		is its normalization.
	\end{lemma}
	
	\begin{proof}
		Set
		$L_C\colonequals\varphi_C^*\mathcal O_{\mathbb P^1}(1)
		\in W_r^1(C)$.
		By Lemma~\ref{lem:gon_norm} and Lemma \ref{lem:gon_pencil},
		$L_B\colonequals\operatorname{Nm}_f(L_C)$ is base-point-free
		with $h^0(B,L_B)=2$.
		Choose a basis $s_0,s_1$ of $H^0(B,L_B)$, which defines a surjective morphism
		\[
		\varphi_B:B\to\mathbb P^1,\qquad
		x\mapsto[s_0(x):s_1(x)]
		\]
		of degree $r$.
		
		Let $u_0,u_1$ be the basis of $H^0(C,L_C)$ defining
		$\varphi_C$. By the construction of the norm map, we can write
		\[
		\operatorname{Nm}_f(bu_0-au_1)
		=F_0(a,b)s_0+F_1(a,b)s_1,
		\]
		where $F_0,F_1\in\mathbf{k}[a,b]$ are homogeneous polynomials
		of degree $\deg(f)$. For every $[a:b]\in\mathbb P^1$, the section $bu_0-au_1$
		is nonzero, and so is its norm. Hence $F_0$ and $F_1$ have no common zero on $\mathbb P^1$.
		We define
		\[
		g:\mathbb P^1\to\mathbb P^1,\qquad
		[a:b]\mapsto[-F_1(a,b):F_0(a,b)].
		\]
		
		For $x=[a:b]\in\mathbb P^1$, the zero divisor of
		$bu_0-au_1$ is $\varphi_C^*(x)$. Thus
		\[
		\begin{aligned}
			f_*\varphi_C^*(x)
			&=\operatorname{div}\bigl(\operatorname{Nm}_f(bu_0-au_1)\bigr)\\
			&=\operatorname{div}\bigl(F_0(a,b)s_0+F_1(a,b)s_1\bigr)
			=\varphi_B^*(g(x)).
		\end{aligned}
		\]
		For any $y\in C$, taking $x=\varphi_C(y)$ shows that
		$f(y)\in \varphi_B^{-1}(g(x))$. Therefore $\varphi_B(f(y))=g(\varphi_C(y))$, and hence
		$\varphi_B\circ f=g\circ\varphi_C$.
		
		Set $X\colonequals B\times_{\mathbb P^1,g}\mathbb P^1$,
		and let $p_2:X\to\mathbb P^1$ be the second projection.
		The commutative diagram induces a finite morphism
		\[
		h\colonequals(f,\varphi_C):C\longrightarrow X,
		\qquad p_2\circ h=\varphi_C.
		\]
		Both $\varphi_C$ and $p_2$ are finite flat of degree $r$.
		
		We claim that $X$ is integral and $h$ is birational.
		For general $x\in\mathbb P^1$, the fibers
		$\varphi_C^{-1}(x)$ and $\varphi_B^{-1}(g(x))$
		each consist of $r$ distinct points.
		The equality $f_*\varphi_C^*(x)=\varphi_B^*(g(x))$
		shows that $f$ induces a bijection between these fibers.
		It follows that there exists a nonempty open subset
		$U\subset\mathbb P^1$ such that $h$ restricts to an isomorphism
		$\varphi_C^{-1}(U)\cra p_2^{-1}(U)$.
		
		Since $p_2^{-1}(U)\simeq\varphi_C^{-1}(U)$ is integral,
		flatness of $p_2$ implies that $X$ is integral.
		The isomorphism over $U$ then shows that $h$ is birational.
		Since $h$ is finite and $C$ is smooth, $h$ is the normalization morphism of $X$.
	\end{proof}
	
	\begin{lemma}\label{lem:gon_tower}
		Let $f_n:C_{n+1}\to C_n$ be surjective morphisms of smooth
		projective curves over $\mathbf{k}$ with $\mr{gon}(C_n)=r$
		for all $n\geq1$.
		Set $f_{1,n}\colonequals f_1\circ\cdots\circ f_{n-1}$,
		with $f_{1,1}=\operatorname{id}$.
		Then there exist commutative diagrams
		\[
		\begin{tikzcd}
			C_n \arrow[r, "\varphi_n"] \arrow[d, "f_{1,n}"']
			& \mathbb P^1 \arrow[d, "g_n"] \\
			C_1 \arrow[r, "\varphi_1"']
			& \mathbb P^1
		\end{tikzcd}
		\]
		with $\deg(\varphi_n)=r$ for all $n\geq1$.
		Moreover, each fiber product
		$C_1\times_{\mathbb P^1,g_n}\mathbb P^1$ is integral,
		and the induced morphism $(f_{1,n},\varphi_n)$ is its
		normalization morphism.
	\end{lemma}
	
	\begin{proof}
		Each $W_r^1(C_n)$ is nonempty closed subset of $\mr{Pic}^r(C_n)$. By Lemma~\ref{lem:gon_norm}, the norm maps restrict to morphisms $\operatorname{Nm}_{f_n}:W_r^1(C_{n+1})\to W_r^1(C_n)$. Hence the subsets
		\[
		Z_n\colonequals\operatorname{Nm}_{f_{1,n}}(W_r^1(C_n))\subset W_r^1(C_1)
		\]
		are nonempty and closed. Moreover, $Z_{n+1}\subseteq Z_n$ for all $n\geq 1$. Since $W_r^1(C_1)$ is Noetherian, the sequence $Z_1\supset Z_2\supset\cdots$  stabilizes and therefore  $\bigcap\limits_{n\geq1}Z_n$ is nonempty.
		
		Choose $L_n\in W_r^1(C_n)$ with $L_1\in\bigcap\limits_{n\geq1}Z_n$, and $\operatorname{Nm}_{f_{1,n}}(L_n)\simeq L_1$. By Lemma~\ref{lem:gon_pencil}, each $L_n$ is base-point-free
		with $h^0(C_n,L_n)=2$. Choose a basis of $H^0(C_n,L_n)$ defining a 
		morphism $\varphi_n:C_n\to\mathbb P^1$ of degree $r$.
		
		For each $n$, apply Lemma~\ref{lem:gon_cartesian} to
		$f_{1,n}$ and $\varphi_n$. This gives morphisms $\psi_n:C_1\to\mathbb P^1$ and
		$g_n:\mathbb P^1\to\mathbb P^1$ satisfying
		$\psi_n\circ f_{1,n}=g_n\circ\varphi_n$. Since $\operatorname{Nm}_{f_{1,n}}(L_n)\simeq L_1$,
		we may compose both $\psi_n$ and $g_n$ with a suitable
		automorphism of $\mathbb P^1$ so that
		$\psi_n=\varphi_1$ for all $n$.
		The last assertion follows directly from Lemma~\ref{lem:gon_cartesian}.
	\end{proof}
	
	\subsection{Weil restriction}
	Let $F/E$ be a finite extension of finitely generated
	fields over $\mathbb Q$, and let $A$ be an abelian variety
	over $F$. Recall, the Weil restriction $R\colonequals\operatorname{Res}_{F/E}(A)$
	is an abelian variety over $E$ of dimension $[F:E]\dim A$.
	There is a natural homomorphism $\lambda:R_F\to A$
	such that, for every $E$-scheme $T$, the map
	\[
	\operatorname{Hom}_E(T,R)
	\longrightarrow
	\operatorname{Hom}_F(T\times_E F,A),
	\qquad
	u\longmapsto\lambda\circ u_F
	\]
	is bijective, where $u_F$ denotes the base change of $u$ to $F$. For an algebraic closure $\overline E$ of $E$,
	there is an isomorphism
	\[
	R_{\overline E}\simeq
	\prod_{\sigma:F\hookrightarrow\overline E}A^\sigma,
	\]
	where the product runs over all $E$-embeddings
	$\sigma:F\hookrightarrow\overline E$, and
	$A^\sigma\colonequals A\otimes_{F,\sigma}\overline E$.
	
	Fix an ample symmetric line bundle $L\in\mr{Pic}(A)$
	and a big and nef adelic line bundle
	$\overline H\in\widehat{\mr{Pic}}(E)$.
	Denote by $\overline H_F\in \widehat{\mr{Pic}}(F)$ its pullback to $F$.
	
	\begin{lemma}\label{lem:ht_weil}
		There exists an ample symmetric line bundle $M$ on $R$
		such that, for every finite extension $E'/E$ with
		$F'\colonequals F\otimes_E E'$ a field
		and every $y\in R(E')$, the point
		$x\colonequals\lambda(y)\in A(F')$ satisfies
		\[
		\widehat h_M^{\overline H}(y)=\widehat h_L^{\overline H_F}(x).
		\]
	\end{lemma}
	
	\begin{proof}
		Let $f:R_F\to R$ be the natural projection, which is finite flat of degree $[F:E]$, and set
		\[
		M\colonequals\operatorname{Nm}_f(\lambda^*L).
		\]
		For each $E$-embedding $\sigma:F\hookrightarrow\overline E$,
		let $\operatorname{pr}_\sigma:R_{\overline E}\to A^\sigma$
		be the corresponding projection, then
		\[
		M_{\overline E}\simeq
		\bigotimes_\sigma\operatorname{pr}_\sigma^*L^\sigma,
		\]
		where $L^\sigma$ is the base change of $L$ along $\sigma$. Thus $M$ is ample and symmetric.
		
		Let $\overline L^{\mathrm{can}}$ be the canonical adelic
		extension of $L$, and set
		$\overline M\colonequals
		\langle \lambda^*\overline L^{\mathrm{can}}\rangle_{R_F/R}\in \widehat{\mr{Pic}}(R)_{\mr{nef}}$.
		We have $[2]^*\overline M\simeq\overline M^{\otimes4}$. Hence $\overline M$ is the canonical adelic extension of $M$.
		
		Let $y\in R(E')$ and $x\colonequals\lambda(y)\in A(F')$,
		where $F'\colonequals F\otimes_E E'$ is a field. Let $y_F:\Spe F'\to R_F$ be the base change of $y:\Spe E'\to R$ along $E\hookrightarrow F$. Then $\lambda\circ y_F=x$.
		By \cite[Theorem~4.1.3]{YZ26}, we have
		\[
		y^*\overline M=y^*\langle\lambda^*\overline L^{\mathrm{can}}\rangle_{R_F/R}
		\simeq\langle y_F^*\lambda^*\overline L^{\mathrm{can}}\rangle_{F'/E'}=
		\langle x^*\overline L^{\mathrm{can}}\rangle_{F'/E'}.
		\]
		Set $d\colonequals\operatorname{trdeg}_{\mathbb Q}E$. Since $[F':F]=[E':E]$, by \cite[Lemma 4.6.1(1)]{YZ26}
		\[
		\widehat h_M^{\overline H}(y)=
		\frac{\langle y^*\overline M\rangle_{E'/E}
			\cdot\overline H^{\,d}}{[E':E]}=
		\frac{\langle x^*\overline L^{\mathrm{can}}\rangle_{F'/E}
			\cdot\overline H^{\,d}}{[F':F]}=
		\frac{\langle x^*\overline L^{\mathrm{can}}\rangle_{F'/F}
			\cdot\overline H_F^{\,d}}{[F':F]}
		=
		\widehat h_L^{\overline H_F}(x).
		\]
		This completes the proof.
	\end{proof}
	
	\subsection{Uniform boundedness of algebraic points}
	In this subsection, let $k$ be a finitely generated
	field over $\mathbb Q$, and let $B$ be a smooth, geometrically connected curve over $k$.
	Set $K\colonequals k(B)$. In particular, $k$ is algebraically closed in $K$. 
	Let $A$ be an abelian variety over $K$ of dimension $g$.
	
	Given an ample and symmetric line bundle $L \in \operatorname{Pic}(A)$ and a big and nef adelic line bundle $\overline{H}$ on $B$, we may define the Néron--Tate height
	\[
	\widehat{h}_L^{\overline{H}} : A(\overline{K}) \longrightarrow \mathbb{R}
	\]
	as defined in Section~\ref{Sec:ht}.
	
	Using Theorem~\ref{thm:bogotype} and Weil restriction, we may prove the following theorem about the gonality.
	
	\begin{theorem}\label{thm:gon_tower}
		With the notation above, assume $A$ contains no nontrivial isotrivial abelian subvariety. Let $(x_n)_{n\geq1}$ be a small sequence of distinct points in $A(\overline K)$ such that
		\[
		K(x_1)\subseteq K(x_2)\subseteq\cdots.
		\]
		Let $C_n$ be the smooth projective curve over
		$\overline k$ with function field $\overline kK(x_n)$.
		Then $\mr{gon}(C_n)\to \infty$ as $n\to \infty$.
	\end{theorem}
	
	\begin{proof}
		Set $K_n\colonequals K(x_n)$. The inclusions
		$\overline kK_n\subseteq\overline kK_{n+1}$
		induce finite morphisms $C_{n+1}\to C_n$. 
		By \cite[Proposition~A.1(vii)]{Poo07}, we have
		\[
		\mr{gon}(C_n)\leq\mr{gon}(C_{n+1})
		\]
		for every $n\geq1$. Assume these gonalities are bounded,
		after removing finitely many terms, we may assume that
		$\mr{gon}(C_n)=r$ for every $n$.
		
		Write $f_{1,n}:C_n\to C_1$ for the induced morphism.
		By Lemma~\ref{lem:gon_tower}, there are commutative diagrams
		\[
		\begin{tikzcd}
			C_n \arrow[r, "\varphi_n"] \arrow[d, "f_{1,n}"']
			& \mathbb P^1_{\overline k} \arrow[d, "g_n"] \\
			C_1 \arrow[r, "\varphi_1"']
			& \mathbb P^1_{\overline k}
		\end{tikzcd}
		\]
		with $\deg(\varphi_n)=r$ for every $n$.
		Moreover, each fiber product $C_1\times_{\mathbb P^1_{\overline k},g_n}
		\mathbb P^1_{\overline k}$ is integral, and the induced morphism $C_n\to C_1\times_{\mathbb P^1_{\overline k},g_n}
		\mathbb P^1_{\overline k}$ is the normalization morphism.
		
		For each $n$, choose a finite extension $k_n/k$
		containing the algebraic closure of $k$ in $K_n$
		such that $C_n,f_{1,n},\varphi_n$, and $g_n$ are
		defined over $k_n$.
		Enlarging these fields if necessary, we may assume that
		$k_n\subseteq k_{n+1}$ for every $n$. 
		Set $F_n\colonequals k_nK_n$ and $E_n\colonequals k_n(t_n)$, where $t_n$ is transcendental over $k_n$. The morphism $\varphi_n$ induces an inclusion
		$E_n\hookrightarrow F_n$ with $[F_n:E_n]=r$.
		
		Set $F\colonequals F_1$ and $E\colonequals E_1$. For each $n$, the morphism
		$g_n:\mathbb P^1_{k_n}\to\mathbb P^1_{k_n}$
		induces an inclusion
		\[
		E=k_1(t_1)\hookrightarrow k_n(t_1)
		\xrightarrow{g_n^*}k_n(t_n)=E_n.
		\]
		By the commutativity of the diagram, this inclusion
		is compatible with $F\hookrightarrow F_n$. Since $C_n$ is the normalization of $C_1\times_{\mathbb P^1_{\overline k},g_n}
		\mathbb P^1_{\overline k}$, we obtain $F\otimes_E E_n\simeq F_n$.
		
		Let $R\colonequals\operatorname{Res}_{F/E}(A_F)$ be the Weil restriction, which admits a natural homomorphism $\lambda:R_F\to A_F$. We have the decomposition
		\[R_{\ov{E}}\simeq \prod_{\sigma:F\hookrightarrow\ov{E}}A_{F}^{\sigma}.\]
		Since $A$ contains no nontrivial isotrivial abelian subvariety, $R$ also contains no nontrivial isotrivial abelian subvariety. The defining property of Weil restriction gives
		natural bijections
		\[
		R(E_n)\simeq A_F(F\otimes_E E_n)\simeq A(F_n).
		\]
		Under these bijections, $x_n\in A(F_n)$ corresponds
		to a point $y_n\in R(E_n)$ satisfying
		$\lambda(y_n)=x_n$.
		
		Let $C_{y_n}$ be a smooth projective curve over $k_1$ with function field $E(y_n)$. Since $E(y_n)\subset E_n=k_n(t_n)$, every connected component of $(C_{y_n})_{\overline{k}}$ is dominated by $\mb{P}^1_{\ov{k}}$. Hence $g(C_{y_n})=0$.
		
		Choose a big and nef adelic line bundle
		$\overline H'\in\widehat{\mr{Pic}}(E)$.
		Denote by $\overline H_F'$ and $\overline H_F$ the
		pullbacks of $\overline H'$ and $\overline H$ to $F$, respectively.
		Applying Lemma~\ref{lem:ht_weil} to $A_F$ and $L_F$,
		we obtain an ample symmetric line bundle
		$M\in\mr{Pic}(R)$.
		By Lemma~\ref{lem:ht_weil} and
		Proposition~\ref{prop:ht}(4), there is a constant
		$c>0$, independent of $n$, such that
		\[
		\widehat h_M^{\overline H'}(y_n)=\widehat h_{L_F}^{\overline H_F'}(x_n)\leq c\,\widehat h_{L_F}^{\overline H_F}(x_n)=c[F:K]\,\widehat h_L^{\overline H}(x_n).
		\]
		Thus $(y_n)_{n\geq1}$ is small. By Theorem~\ref{thm:bogotype}, the points $y_n$ take
		only finitely many values. So $x_n=\lambda(y_n)$ take only finitely many values. However, $(x_n)_{n\geq 1}$ is a sequence of distinct points. Contradiction! This completes the proof.
	\end{proof}
	
	Recall, for a point $x\in A(\ov{K})$, denote $C_x^{\circ}$ the smooth projective curve over $\ov{k}$ with function field $\ov{k}K(x)$. For each positive integer $n$, let
	\[
	\operatorname{gon}(n)\colonequals
	\min_{x\in A[n]^\times}\operatorname{gon}(C_x^\circ).
	\]
	
	\begin{corollary}\label{cor:infgon}
		Let $k$ be a finitely generated field over $\mb{Q}$, $B$ be a smooth geometrically connected curve over $k$, and $A$ be an abelian variety over $K:=k(B)$. Assume $A$ contains no non-trivial isotrivial abelian subvariety.Then, for every prime $\ell$,
		\[
		\lim_{n\to\infty}\operatorname{gon}(\ell^n)=+\infty.
		\]
	\end{corollary}
	
	\begin{proof}
		For each $n$, consider the set \[
		T_n\colonequals
		\left\{
		x\in A[\ell^n]^\times:
		\operatorname{gon}(C_x^\circ)\leq D.
		\right\}
		\]
		For $x\in T_{n+1}$, the inclusion
		$\overline kK(\ell x)\subseteq\overline kK(x)$ induces
		a finite morphism $C_x^\circ\to C_{\ell x}^\circ$, so
		\[
		\operatorname{gon}(C_{\ell x}^\circ)
		\leq\operatorname{gon}(C_x^\circ).
		\]
		Thus $[\ell]T_{n+1}\subseteq T_n$.
		
		Now, assume that $\operatorname{gon}(\ell^n)$ does not tend
		to infinity. Then there exists an integer $D\geq1$ such
		that $T_n\neq \varnothing$ for all $n$.
		
		Since each $T_n$ is finite and $[\ell]T_{n+1}\subseteq T_n$,
		we may successively choose $x_n\in T_n$ with
		$\ell x_{n+1}=x_n$. The sequence $(x_n)_{n\geq 1}$ satisfies
		\[
		K(x_1)\subseteq K(x_2)\subseteq\cdots
		\]
		and $\widehat h_L^{\overline H}(x_n)=0$ for all $n$. By Theorem~\ref{thm:gon_tower}, we have $\operatorname{gon}(C_{x_n}^\circ)\to\infty$, contradiction. This completes the proof.
	\end{proof}

	To strengthen Theorem~\ref{thm:sec1} and Theorem~\ref{thm:CT1},
	we use the Faltings--Frey theorem \cite[Proposition 2]{Fre94}, which extends
	the finiteness statement in Mordell's conjecture to points of bounded degree. 
	
	\begin{theorem}[Faltings--Frey]\label{thm:faltings_frey}
		Let $k$ be a finitely generated field over $\mathbb Q$,
		and let $C$ be a smooth projective geometrically connected curve over $k$. Let $d\geq1$ be an integer. If $\mr{gon}(C_{\overline k})>2d$, then the set
		\[
		\{x\in C(\overline k):[k(x):k]\leq d\}
		\]
		is finite.
	\end{theorem}
	
	\begin{remark}
		Frey’s result was originally stated for curves over number fields, but the same argument applies to curves over fields finitely generated over $\mathbb Q$, since Faltings’ theorem holds in this setting.
	\end{remark}
	
	\begin{proof}[Proof of Theorem~\ref{thm:sec_strong}]
		Let $x_0\in A(K)$ be the point in the generic fiber
		corresponding to the section $P$. For each $x\in A(\ov{K})$, let
		$C_x\colonequals\Zar{\{x\}}\subset\mc{A}$, and let $C_x^\circ$ be the smooth projective curve over $\ov{k}$ with function field $\ov{k}K(x)$. Then $C_x^\circ$ is the smooth projective model of a connected component of $(C_x)_{\ov{k}}$. Consider the set
		\[
		T=\{x_0\}\cup
		\left\{
		x\in A(\overline K):
		\ell^n x=x_0\text{ for some }n\geq0,\quad
		\mr{gon}(C_x^\circ)\leq2d
		\right\}.
		\]
		If $T$ is infinite, we may choose a sequence of distinct points $(x_n)_{n\geq0}$ in $T$ such that $\ell x_{n+1}=x_n$ for all $n\geq0$. This implies 
		\[ \lim_{n \to \infty} \widehat{h}_L^{\overline{H}}(x_n) = \lim_{n \to \infty} \frac{1}{\ell^{2n}} \widehat{h}_L^{\overline{H}}(x_0) = 0. \]
		Note that $K(x_n)\subset K(x_{n+1})$ for all $n\geq 1$. Then Theorem~\ref{thm:gon_tower} implies that $\mr{gon}(C_{x_n}^{\circ})\to \infty$. Contradiction. Thus, $T$ is finite.

		Now define 
		\[S \colonequals  \{ x \in A(\overline{K}) : \ell x \in T \text{ and } x \notin T \}.\] Since $T$ is finite, $S$ is also finite. By the definition of $T$, every $x \in S$ must satisfy $\mr{gon}(C_x^{\circ})>2d$. Moreover, each $C_x$ is a connected component of
		$[\ell^n]^{-1}(P)$ for some $n$, so $C_x\to B$ is finite \'etale and $C_x$ is smooth.
		
		For $x\in S$, set
		\[
		C_x(k;d)\colonequals
		\{z\in C_x(\overline k):[k(z):k]\leq d\}.
		\]
		By Theorem~\ref{thm:faltings_frey}, this set is finite.
		Indeed, if $C_x$ is not geometrically connected, we view it as a geometrically connected curve over
		the algebraic closure of $k$ in $K(x)$ and apply the theorem over that field. Consequently,
		\[
		V\colonequals
		\bigcup_{x\in S}
		\operatorname{Im}\bigl(C_x(k;d)\to B(\overline k)\bigr)
		\]
		is finite.
		
		For $b\in B(\overline k)$ with $[k(b):k]\leq d$,
		write
		\[
		D_b=
		\left\{
		z\in\mathcal A_b(\overline k):
		[k(b,z):k]\leq d,\quad
		\ell^n z=P_b\text{ for some }n\geq0
		\right\}.
		\]
		We claim that, if $b\notin V$, then $D_b\subseteq
		\bigcup_{x\in T}(C_x)_b(\overline k)$. In fact, let $z\in D_b$ with $\ell^n z=P_b$.
		Choose $x\in[\ell^n]^{-1}(x_0)$ such that $z\in C_x$. If $x\notin T$, then $y=\ell^j x\in S$ for some $j\geq0$. Note that 
		\[
		[k(\ell^j z):k]\leq[k(b,z):k]\leq d.
		\]
		It follows that $\ell^j z\in C_y(k;d)$, This implies $b\in V$, which is a contradiction. Therefore, for $b\notin V$,
		\[
		\#D_b
		\leq\sum_{x\in T}\#(C_x)_b(\overline k)
		\leq\sum_{x\in T}[K(x):K].
		\]
		
		For each $b\in V$, we claim that $D_b$ is finite. Indeed, fix a canonical Moriwaki height $\widehat h_b$ on $\mathcal A_b/k(b)$ associated to an ample symmetric line bundle. Every $z\in D_b$ satisfies 
		\[
		\widehat h_b(z)=\ell^{-2n}\widehat h_b(P_b)
		\leq\widehat h_b(P_b)
		\]
		for some $n\geq0$. Since $[k(b,z):k(b)]\leq d$, $D_b$ is finite by the Northcott property.
		
		Set
		\[
		N=
		\max_{b\in V}\#D_b
		+\sum_{x\in T}[K(x):K],
		\]
		which depends only on $\mc{A},P,k,d$ and $\ell$. For every finite extension $k'/k$ with $[k':k]\leq d$
		and every $b\in B(k')$, all points
		$z\in\mathcal A_b(k')$ satisfying
		$\ell^n z=P_b$ belong to $D_b$. The preceding bounds therefore prove the theorem.
	\end{proof}
	
	\begin{theorem}\label{thm:CT_alg}
		Let $k$ be a finitely generated field over $\mathbb Q$,
		let $B$ be a smooth geometrically connected curve over $k$,
		and let $\pi:\mathcal A\to B$ be an abelian scheme. For every integer $d\geq1$ and every prime $\ell$, there exists an integer
		$N=N(\mathcal A,B,k,d,\ell)$ such that, for every
		finite extension $k'/k$ with $[k':k]\leq d$ and every
		$b\in B(k')$,
		\[
		\#\mathcal A_b(k')[\ell^\infty]\leq N,
		\]
		where $\mathcal A_b(k')[\ell^\infty]
		\colonequals
		\bigcup_{n\geq0}\mathcal A_b(k')[\ell^n]$
		is the set of $\ell$-primary torsion points
		in $\mathcal A_b(k')$.
	\end{theorem}
	
	\begin{proof}
		The proof is the same as that of Theorem~\ref{thm:CT1},
		with Theorem~\ref{thm:sec} replaced by
		Theorem~\ref{thm:sec_strong}.
	\end{proof}
	\bibliography{reference.bib}
	\bibliographystyle{alpha}

	\noindent \small{Address: \textit{Institute for Theoretical Sciences, Westlake University, Hangzhou 310030, China}}
	
	\noindent \small{Email: \texttt{jizhuchao@westlake.edu.cn}}
	
	\noindent \small{Address: \textit{School of Mathematical Sciences, Peking University, Beijing 100871, China}}
	
	\noindent \small{Email: \texttt{soyo999@pku.edu.cn}}
	
	\noindent \small{Address: \textit{Beijing International Center for Mathematical Research, Peking University, Beijing 100871, China}}
	
	\noindent \small{Email: \texttt{xiejunyi@bicmr.pku.edu.cn}}
\end{document}